\documentclass{amsart}
\usepackage{graphicx}
\usepackage{amssymb}
\usepackage{amsthm}
\usepackage{array}

\theoremstyle{plain}
\newtheorem{theorem}{Theorem}[section]
\newtheorem{lemma}[theorem]{Lemma}
\newtheorem{prop}[theorem]{Proposition}
\newtheorem{cor}[theorem]{Corollary}

\theoremstyle{definition}
\newtheorem{defn}[theorem]{Definition}

\newtheorem{ques}[theorem]{Question}
\newtheorem{notadef}[theorem]{Notation and Definition}
\newtheorem{rmknota}[theorem]{Notation and Remark}

\newtheorem{rmk}[theorem]{Remark}
\newtheorem{nota}[theorem]{Notation}

\newtheorem{example}[theorem]{Example}

\def\D{\mathcal{D}}

\def\C{\mathcal{C}}
\def\F{\mathcal{F}}
\def\sS{\mathcal{S}}

\def\bQ{\mathbb{Q}}
\def\bZ{\mathbb{Z}}
\def\bF{\mathbb{F}}
\def\bN{\mathbb{N}}

\def\Hom{\mathrm{Hom}}

\def\End{\mathrm{End}}
\def\Res{\mathrm{Res}}
\def\Ker{\mathrm{Ker}}
\def\Im{\mathrm{Im}}

\def\Inf{\mathrm{Inf}}
\def\Ind{\mathrm{Ind}}

\def\Def{\mathrm{Def}}

\def\Ten{\mathrm{Ten}}

\def\Iso{\mathrm{Iso}}

\title[On the unit group of the Burnside ring]{On the unit group of the Burnside ring as a biset functor for some solvable groups}
\author{Jamison Barsotti}
\address{Department of Mathematics,
University of California Santa Cruz, Santa Cruz, CA 95064}
\email{jbarsott@ucsc.edu}
\begin{document}

\begin{abstract}
The theory of bisets has been very useful in progress towards settling the longstanding question of determining units for the Burnside ring. 
In 2006 Bouc used bisets to settle the question for $p$-groups. 
In this paper, we provide a standard basis for the unit group of the Burnside ring for groups that contain a abelian subgroups of index two. 
We then extend this result to groups $G$, where $G$ has a normal subgroup, $N$, 
of odd index, such that $N$ contains an abelian subgroups of index $2$. 
Next, we study the structure of the unit group of the Burnside ring as a biset functor, $B^\times$ on this class of groups and determine its lattice of subfunctors.
We then use this to determine the composition factors of $B^\times$ over this class of groups. 
Additionally, we give a sufficient condition for when the functor $B^\times$, defined on a class of groups closed under subquotients,
has uncountably many subfunctors.
\end{abstract}
\maketitle
\section{Introduction}
If $G$ is a finite group, recall that the unit group of the Burnside ring, denoted by $B^\times(G)$, is always an elementary abelian $2$-group. 
It has been studied extensively by Yoshida 
(\cite{Yo}), Matsuda (\cite{Ma}, \cite{Ma2}), Yal\c{c}{\i}n (\cite{Ya}), and Bouc (\cite{B2}), although its rank for a general finite group is still unknown. In \cite{B2} Bouc succeeded in computing $B^\times(G)$ for the case where $G$ is a $p$-group using the theory of bisets. This is obtained through realizing the unit group of the Burnside ring as a biset functor, $B^\times$, on the biset category for finite groups $\C$ (see Definition~\ref{bisetCategoryDefinition}).

The main results of this paper continue in the tradition of Bouc's results for $B^\times$ on $p$-groups.
 Matsuda had previously determined $B^\times(G)$ for $G$ abelian and $G$ dihedral \cite{Ma}. Theorem~\ref{basisofab2} finds a standard basis for $B^\times(G)$ when $G$ has an abelian subgroup of index $1$ or $2$ and can be thought of as a generalization and merger of Matsuda's
results placed in the context of bisets. 
Corollary~\ref{oddIndexCor} then extends this computation to $B^\times(G)$ when $G$ contains a normal subgroups of odd index, that has an abelian subgroup
of index $1$ or $2$. We then turn our 
attention to Theorem~\ref{subParam}, which computes the lattice of subfunctors of $B^\times$ in the full subcategory of the biset category of finite groups $\C$, consisting of such groups, 
which denote by $\C'$, 
as well as determining its composition factors (Theorem~\ref{compFactorsBX}). The overall goal would be a 
theory of structure and computation for $B^\times$ when $G$ is solvable. Viewed in this way, this paper may be considered as progress in this direction.

If $\D$ is a subcategory of $\C$ that is closed under subquotients, an interesting consequence of Theorem~\ref{subParam} is Corollary~\ref{uncountableSubs}, 
which gives a sufficient condition on $\D$, for when $B^\times$, defined on $\D$, will have uncountably many subfunctors. 
In fact, Theorem~\ref{converseUncountableSubs} 
shows this is a necessary condition if $\D\subset \C'$. It is natural then to ask if we can relax the condition that $\D\subset \C'$ and still
have the same necessary condition for $\D$. If not, what other conditions do we need for the general case?

As an application of these results, Theorem~\ref{simpleComputation} computes the $\bF_2$-dimension for the simple biset functors 
$S_{G,\bF_2}(H)$ for various finite groups $G$ and $H$. In Example~\ref{computation1} and Example~\ref{computation2}, we make these computations explicit 
when $G$ is trivial or certain dihedral groups and $H$ is a dihedral group.
We connect this to the exponential map $\varepsilon_G:B(G)\to B^\times(G)$,
studied in particular by Bouc (\cite{B2}) and Yal\c{c}{\i}n (\cite{Ya}), and determine when it is surjective for the finite groups we focus on in 
Corollary~\ref{surjectiveExpo}. In particular, Corollary~\ref{surjectiveExpoDihedral} gives a numerical characterization for when $\varepsilon_{D_{2n}}$
is surjective for dihedral groups $D_{2n}$.
\\

{\bf Acknowledgments:}
The material in this paper is based on work done towards the author's Ph.D. thesis. The author would like to acknowledge and thank his advisor, Robert 
Boltje, whose patience, support, and encouragement where instrumental in the writing of this paper. 
The author would also like to thank Serge Bouc, who supplied the author with helpful ideas. In particular, it was on Bouc's suggestion that
Theorem~\ref{basisofab2} could be extended to a larger class of groups using Proposition 6.5 in \cite{B2}.

The author would also like to thank the referee for many helpful suggestions. In particular, the suggestion for the name of the the groups in Definition~\ref{pseudodihedralDefinition}.

\section{preliminaries}
Given a finite group $G$, recall the \emph{Burnside ring} of $G$, which we denote by $B(G)$, is defined to be the Grothendieck ring of isomorphism classes of finite (left) $G$-sets, with respect to the operations disjoint union and cartesian product. If we just look at the underlying abelian group structure, we call it the \emph{Burnside group} of $G$. 
If $X$ is a $G$-set, we will denote the image of $X$ in $B(G)$ as $[X]$.
Given a left $G$-set $X$, we can view it as a right $G$-set, via the action $x\cdot g=g^{-1}x$, for all $x\in X$ and $g\in G$, and vice versa. If $G$ and $H$ are both finite groups and $X$ is a left $G$-set and a right $H$-set such that, $(gx)h=g(xh)$ for any $x\in X$, $g\in G$, and $h\in H$, then we say $X$ is a \emph{$(G,H)$-biset}. There is a one-to-one correspondence between $(G,H)$-bisets and left $G\times H$-sets, so we set
the Burnside group of $(G,H)$-bisets to be $B(G,H):=B(G\times H)$.

\begin{rmk}[\cite{B1}, 2.3.9]
{\bf Elementary Bisets.} Let $G$ be a finite group. The following bisets are called \emph{elementary bisets}. We will make use of them frequently.
\begin{itemize}
\item If $H$ is a subgroup of $G$, then the set $G$ is an $(H,G)$-biset given by the usual left and right multiplication in $G$ and is denoted by $\Res_H^G$. 
We call this \emph{restriction} from $G$ to $H$.
\item If $H$ is a subgroup of $G$, then the set $G$ is a $(G,H)$-biset given by the usual left and right multiplication in $G$. It is denoted by $\Ind_H^G$. 
We call this \emph{induction} from $H$ to $G$.
\item If $N$ is a normal subgroup of $G$, then the set $G/N$ is a $(G,G/N)$-biset given by the usual multiplication on the right and projection to $G/N$ then multiplication on the left. It is denoted 
$\Inf_{G/N}^G$. We call this \emph{inflation} from $G/N$ to $G$.
\item If $N$ is a normal subgroup of $G$, then the set $G/N$ is a $(G/N,G)$-biset given by the usual multiplication on the left and projection to $G/N$ then multiplication on the right. It is denoted 
$\Def_{G/N}^G$. We call this \emph{deflation} from $G$ to $G/N$.
\item If $f:G\to H$ is a group isomorphism, then the set $H$ is an $(H,G)$-biset, where the left action is usual multiplication, and the right action is given by multiplication of the image through $f$, denoted by $\Iso(f)$ or $\Iso_H^G$ if the isomorphism is clear from the context.
\end{itemize}
We abusively use the same notation for the images of elementary bisets in their respective Burnside groups.
\end{rmk}

The elementary bisets satisfy various relations which are summarized in \cite{B1}. We list a few of the ones pertinent to this paper for reference.

\begin{rmk}(\cite{B1}, 1.1.3)\label{ebisetTranCom}
\hfill
\begin{enumerate}
\item \emph{Transitivity of relations:}
\begin{enumerate}
\item If $K\leq H\leq G$, then 
\[\Res^H_K\circ \Res^G_H=\Res^G_K\in B(K,G),\]
\[\Ind^G_H\circ \Ind^H_K=\Ind^G_K\in B(G,K).\]
\item If $N$ and $M$ are normal subgroups of $G$ with $N\leq M$, then
\[\Inf^G_{G/N}\circ \Inf^{G/N}_{G/M}=\Inf^G_{G/M}\in B(G,G/M),\]
\[\Def^{G/N}_{G/M}\circ \Def^{G}_{G/N}=\Def^G_{G/M}\in B(G/M,G).\]
Here, $\Inf_{G/M}^{G/N}:=\Inf^{G/N}_{(G/N)/(M/N)}\circ\Iso(\alpha^{-1})$ and $\Def_{G/M}^{G/N}:=\Iso(\alpha)\circ\Def^{G/N}_{(G/N)/(M/N)}$
where $\alpha:(G/N)/(M/N)\to G/M$ is the canonical isomorphism.

\end{enumerate}
\item \emph{Commutativity relations:}
\begin{enumerate}
\item (Mackey formula) If $H$ and $K$ are subgroups of $G$, then
\[\Res_H^G\circ\Ind_K^G=\sum_{x\in[H\backslash G/K]}\Ind^H_{H\cap\,^xK}\circ\Iso(\gamma_x)\circ\Res^K_{H^x\cap K},\]
where $[H\backslash G/K]$ is a set of representatives of $(H,K)$-double cosets in $G$, and $\gamma_x:H^x\cap K\to H\cap\,^xK$ is the group isomorphism induced by conjugation by $x$.
\item If $H$ is a subgroup of $G$ and $N$ is a normal subgroup of $G$ such that $N\leqslant H$, then
\[\Ind_H^G\circ\Inf_{H/N}^H=\Inf^G_{G/N}\circ\Ind_{H/N}^{G/N}.\]
\end{enumerate}
\end{enumerate}
\end{rmk}

Let $G$, $H$, and $K$ be finite groups, $X$ a $(G,H)$-biset, and $Y$ is an $(H,K)$-biset. We set $X\times_HY$ to be the set of $H$-obits of the $H$-set $X\times Y$ with the action
$h\cdot(x,y)=(xh^{-1},hy)$,  for any $h\in H$ and $(x,y)\in X\times Y$.
We then let $X\times_HY$ take on the natural action from $G\times K$. Elements of $X\times_HY$ are denoted by $(x,_Hy)$. The operation $\times_H$ induces a bilinear map
from $B(G,H)\times B(H,K)\to B(G,K)$.

\begin{defn}[\cite{B1}, 3.1.1]\label{bisetCategoryDefinition}
The \emph{biset category $\C$} of finite groups is the category defined as follows:
\begin{itemize}
\item The objects of $\C$ are finite groups.
\item If $G$ and $H$ are finite groups, then $\Hom_\C(G,H)=B(H,G)$.
\item If $G, H$, and $K$ are finite groups, then the composition $v\circ u$ of morphisms $u\in \Hom_\C(G,H)$ and $v\in \Hom_\C(H,K)$ is equal to $v\times_Hu$.
\item For any finite group $G$, the identity morphism of $G$ in $\C$ is equal to $[G].$
\end{itemize}
\end{defn}

\begin{rmk}\label{elementaryBisetPresentation}
The elementary bisets induce a presentation of $\C$, (\cite{B1}, Remark $3.1.2$). In particular, given any two objects $G,H\in \C$,  for any $\alpha\in B(G,H)$
that is the image of a transitive $(G, H)$-biset, there is a subquotient $A/B$ of $G$ and a subquotient $C/D$ of $H$, and an isomorphism $f:C/D\cong A/B$ such that
\[\alpha=\Ind^G_A\circ\Inf^A_{A/B}\circ\Iso(f)\circ\Def_{C/D}^D\circ\Res_{C}^H.\] 

\end{rmk}

If $R$ is a commutative ring with identity and $\D$ is a preadditive subcategory of $\C$, then the category $R\D$ is the category whose objects are the same as $\D$ and if
$G$ and $H$ are two objects in $R\D$, then $\Hom_{R\D}(G,H)=R\otimes_{\bZ}\Hom_{\D}(G,H)$. A \emph{biset functor} on $\D$ with values in $_R\text{Mod}$ is an $R$-linear functor
from $R\D$ to $_R\text{Mod}$ (\cite{B1}, 43). The category of biset functors on $\D$ with values in $_R\text{Mod}$ whose morphisms are natural transformations is denoted by $\F_{\D,R}$. It is straightforward to verify that $\F_{\D,R}$ is an abelian category. For more information, see Proposition $3.2.8$ of \cite{B1}.

For the rest of the paper, we will mostly consider subcategories $\D\subset \C$ that satisfy the following definition.

\begin{defn}[\cite{B1}, 4.1.7, page 55]
A class $\D$ of finite groups is said to be closed under taking subquotients if any group isomorphic to a subquotient of an object of $\D$ is in $\D$.

A subcategory $\D$ of $\C$ is said to be  \emph{replete} if it is a full subcategory whose class of objects is closed under taking subquotients.
\end{defn}

\subsection{Simple biset functors}
\hfill\\

Suppose $R$ is commutative ring with identity and $\D$ is a preadditive subcategory. If $F$ is an object in $\F_{\D,R}$, (i.e. a biset functor on $\D$ with values in $\,_R\text{Mod}$)
we say that $F$ is a \emph{simple biset functor} if it is a simple object in $\F_{\D,R}$.

\begin{prop}[\cite{B1}, 4.2.2]\label{resSimpOrZero}
Let $R$ be a commutative ring with identity and let $\D$ be a replete subcategory of $\C$, let $\D'$ be a full subcategory of $\D$.\\
If $F$ is a simple object of $\F_{\D,R}$, then the restriction of $\F$ to $R\D'$ is either the zero object or a simple object of $\F_{\D',R}$.
\end{prop}

Recall that the \emph{double Burnside ring} of $G$ over $R$ is the endomorphism algebra $E_G:=\End_{R\D}(G)=RB(G,G)$.

\begin{prop}[\cite{B1}, 4.3.2]
Let $R$ be a commutative ring with identity, and $\D$ a replete subcategory of $\C$. If $G$ is an object of $\D$, denote $I_G$ to be the $R$-submodule of $E_G$ generated
by all endomorphisms of $G$ which can be factored through some object $H$ of $\D$ with $|H|<|G|$. Then $I_G$ is a two sided ideal of $E_G$, and there is a decomposition
\[E_G=A_G\oplus I_G\]
where $A_G$ is an $R$-subalgebra, isomorphic to the group algebra $R\mathrm{Out}(G)$ of the group of outer automorphisms of $G$.
\end{prop}

\begin{rmk}\label{simpleParam}
We briefly recall the parametrization of simple biset functor on a replete subcategory $\D$ of $\C$. Full details can be found in chapter $4$ of \cite{B1}:\\
If $S$ is a simple object of $\F_{\D,R}$, then there is a minimal object $G$ of $\D$, with respect to order, such that $S(G)\neq \{0\}$ is a simple $E_G/I_G\cong R\text{Out}(G)$-module, 
and for any object $H$ in $\D$ such that $S(H)\neq\{0\}$, $G$ is isomorphic to a subquotient of $H$. Conversely, if $G$ is an object of $\D$ and $V$ is a simple $R\text{Out}(G)$-module, then
there is a simple object of $\F_{\D,R}$, which we denote by $S_{G,V}$, such that $S_{G,V}(G)\cong V$ and for any object $H$ in $\D$, $S_{G,V}(H)\neq 0$ implies that $G$ is isomorphic to a 
subquotient of $H$. Two simple objects in $\F_{\D,R}$, $S_{G,V}$ and $S_{G',V'}$ are isomorphic if and only if there is a group isomorphism $\varphi:G\to G'$ and an $R$-module isomorphism 
$\psi:V\to V'$ such that for all $v\in V$, and all $a\in \text{Out}(G)$, $\psi(a\cdot v)=(\varphi a\varphi^{-1})\cdot\psi(v).$
\end{rmk}

\subsection{Composition factors}
\hfill

\begin{defn}
Let $\D$ be a preadditive subcategory of $\C$. Suppose $F$ is an object of $\F_{\D,R}$. A simple functor $S$ is a \emph{composition factor} of $F$ if there are subfunctors 
$F''\subset F'\subset F$ on $R\D$, such that $F'/F''\cong S$.
\end{defn}

If $G$ is a finite group we let $\C_{\sqsubseteq_G}$ denote the full subcategory of $\C$ whose objects are subquotients of $G$.

\begin{defn}[\cite{We}]
Let $G$ be an object of $\D$ and $F$ an object of $\F_{\D,R}$. The functor $F$ has a \emph{composition series over $G$}
if there is a series of subfunctors
\[0=T_0\subseteq B_1\subset T_1\subseteq\cdots \subseteq B_m\subset T_m\subseteq B_{m+1}=F\]
such that
\begin{itemize}
\item $T_i/B_i$ is a simple functor, whose restriction to $\C_{\sqsubseteq_G}$ is nonzero for all $i=1\dots,m$.
\item $\Res^\D_{\C_{\sqsubseteq_G}}(B_{i+1}/T_i)=0$ for all $i=0,\dots,m$.
\end{itemize}

If such a composition series exists, we will call the set of simple functors $T_i/B_i$ together with their multiplicities the
\emph{composition factors of $F$ over $G$}.
\end{defn}

\begin{prop}[\cite{We}, Proposition 3.1]
Let $G$ be a fixed object of $R\D$ and $F$ a biset functor on $R\D$. If $F$ has a composition series over $G$, then any other composition series over $G$ of $F$ has the same length and the composition factors over $G$ (taken with multiplicities)
are the same.
\end{prop}

\begin{theorem}[\cite{Ba}, Theorem 2.42]
\label{compSeriesExist}
If $R$ is a field then every biset functor on $R\D$ with values in the category of finitely generated $R$-modules $\,_R\mathrm{mod}$ 
has a composition series over $G$, for all objects $G$ in $R\D$.
\end{theorem}

Suppose $R$ is a field and $F$ is a biset functor on $\D$ with values in $\,_R\text{mod}$. If $S$ is a composition factor of $F$, then there is an object $G$ of $R\D$ and simple $R\text{Out}(G)$-
module $V$, such that $S\cong S_{G,V}$. It follows from Theorem~\ref{compSeriesExist} that $S$ will be a composition factor of $F$ over $G$. We 
can then define the \emph{multiplicity of $S$} as a composition factor of $F$ to be the number of times it shows up as a composition factor of $F$ over 
$G$. Note that choosing a different parametrization for $S\cong S_{G',V'}$ would result in the same multiplicity, since $G\cong G'$ and thus 
$\C_{\sqsubseteq_G}$ and $\C_{\sqsubseteq_{G'}}$ are the same subcategory of $\C$.

\section{Idempotents of RB(G,G)}
For this section, we will assume $\D$ is a replete subcategory of $\C$.
 
\begin{defn}[\cite{B1}, Definition 6.2.4]
Let $G$ be a finite group. If $N$ is a normal subgroup of $G$, let $f^G_N$ denote the element of $RB(G,G)$ defined by
\[f^G_N=\sum_{N\leq M\unlhd G}\mu_{\unlhd G}(N,M)[\Inf^G_{G/N}\times_{G/N}\Def_{G/N}^G]\]
where $\mu_{\unlhd G}$ is the M{\"o}bius function of the poset of normal subgroups of $G$.
\end{defn}

\begin{prop}[\cite{B1}, Proposition 6.2.7]
Let $G$ be a finite group. Then the elements $f^G_N$, for $N\unlhd G$, are mutually orthogonal idempotents of $RB(G,G)$ and
\[\sum_{N\unlhd G}f_N^G=\mathrm{Id}_G.\]
\end{prop}

\begin{nota}[\cite{B1}, Definition and Notation 6.3.1]
Let $F$ be a biset functor, and let $G$ be an object of $\D$. The set of \emph{faithful elements} of $F(G)$ is the $R$-submodule
\[\partial F(G)=f^G_1F(G)\]
of $F(G)$.
\end{nota}

If $F$ is a biset functor, then for any object $G$ of $\D$, we can use these idempotents to decompose $F(G)$ as an $R$-module. In particular, we have the following results.

\begin{lemma}[\cite{B1}, Proposition 6.3.2]
\label{faithEltChar}
Let $F$ be a biset functor and $G$ an object of $\D$.
\begin{enumerate}
\item If $1<N\unlhd G$, and $u\in F(G/N)$, then $f_1^G\Inf_{G/N}^Gu=0$.
\item $F(G)=\partial F(G)\oplus \sum_{1<N\unlhd G}\Im\Inf_{G/N}^G$.
\item $\partial F(G)=\bigcap_{1<N\unlhd G}\Ker\Def_{G/N}^G.$
\end{enumerate}
\end{lemma}

\begin{prop}[\cite{B1}, Proposition 6.3.3]
\label{normChar}
Let $F$ be a biset functor, and $G$ be an object of $\D$. Then the map $\delta:F(G)\to\oplus_{N\unlhd G}\partial F(G/N)$
defined by
\[\delta(u)=\bigoplus_{N\unlhd G}f_1^{G/N}\Def^G_{G/N}u\]
is an isomorphism of abelian groups. The inverse isomorphism is the map $\iota$ defined by
\[\iota\left(\bigoplus_{N\unlhd G}v_N\right)=\sum_{N\unlhd G}\Inf_{G/N}^Gv_N.\]
\end{prop}

\section{The unit group of the Burnside ring as a biset functor}

If $R$ is a commutative ring and $\D$ is a replete subcategory, we define the \emph{Burnside functor} on $R\D$ to be $RB:=\Hom_{R\D}(1,-)$. If $R=\bZ$, then $B$ is equivalent to the functor
that takes an object $G$ of $\D$ to its Burnside ring.

Recall that a consequence of Burnside's Theorem is that the Burnside ring of a finite group $G$, $B(G)$, has an embedding into $\bZ^{|\sS_G|}$ with finite cokernel, 
where $\sS_G$ is a representative set of conjugacy classes of subgroups of $G$. This embedding is induced by the map that takes a finite $G$-set $X$, to the column vector 
$(|X^H|)_{H\in \sS_G}$. We abusively use the notation $|a^H|$ when $a\in B(G)$ and $H\in\sS_G$, to denote the value of this embedding in the $H$ entry. This embedding is called the 
\emph{ghost map} of $B(G)$ and its codomain, $Z^{|\sS_G|}$, is called \emph{the ghost ring} of $B(G)$. Throughout this paper, we identify $B(G)$ with its image in the ghost ring. 
When determining elements of $B(G)$, it is often useful to work inside the ghost ring.

It is useful to recall that if $H$ and $S$ are subgroups of $G$, then $|[G/H]^S|=0$ unless $gSg^{-1}\leqslant H$ for some $g\in G$ and
$|[G/H]^H|=[N_G(H):H]$.

One consequence of the embedding of $B(G)$ into its ghost ring is that $B(G)^\times$ is an elementary abelian $2$-group. Although, determining the exact number of elements is difficult, in general. However,
we may view the unit group of the Burnside ring as a biset functor over $\D$.  This gives us various maps between unit groups of Burnside rings.
For the paper at hand we will only need an explicit formulation of how these maps work between ghost rings of Burnside rings of finite groups. The theoretical groundwork for this biset functor can be found in Section $11.2$ of \cite{B1}.

\begin{defn}[\cite{B1}, 11.2.17]
Let $G$ and $H$ be finite groups and $U$ be a finite $(H,G)$-biset. For $u\in U$ and $L\leqslant H$, let 
\[L^u:=\{g\in G|\,\,\exists\,\, l\in L, \,\,l\cdot u=u\cdot g\}.\]
Then $L^u$ is a subgroup of $G$. If $a\in B(G)$, then define
\[T_U(a)=\left(\prod_{u\in[L\backslash U/G]}|a^{L^u}|\right)_{L\in \sS_H}\]
where $[L\backslash U/G]$ is a set of representatives of $L\backslash U/G$.
\end{defn}

\begin{theorem}[\cite{B1}, Corollary 11.2.21]
There exists a unique biset functor $B^\times$ on $\C$ such that $B^\times (G)=B(G)^\times$, for any finite group $G$, and 
\[B^\times(U)=T_U:B^\times(G)\to B^\times(H)\]
for any finite groups $G$ and $H$, and any finite $(H,G)$-biset $U$.
\end{theorem}

We then, of course, restrict $B^\times$ to the replete subcategory $\D$.
It is worth computing the maps $T_U$ when $U$ is an elementary biset.

\begin{rmk}\label{tensorInductionComp}
Suppose $G$ is a finite group, $H$ is a subgroup of $G$, and $N$ is a normal subgroup of $G$. If $U$ is either 
$\Res^G_{H}$ or  $\Inf^G_{G/N}$, then
$T_U$ is equal to the maps $B(U)$ restricted to their respective unit groups. For $U=\Ind^G_H$, we denote $\Ten^G_H:=T_U$. If $L$ is a subgroup of $G$ and
$x\in G$ then, viewing $x$ as a point in the $(G,H)$-biset $G$, we have $L^x=x^{-1}Lx\cap H$ and
\begin{equation}\label{tenInduct}
\Ten^G_H(a)=\left(\prod_{x\in[L\backslash G/H]}|a^{(x^{-1}Lx\cap H)}|\right)_{L\in \sS_G}
\end{equation}
for any $a\in B^\times(H)$. If $V=\Def_{G/N}^G$, then for any $b\in B^\times(G)$ and any subgroup $K/N$ of $G/N$, we have
\[|T_V(b)^{K/N}|=|b^K|=|(b^N)^{K/N}|,\]
where $b^N$ denotes the element in $B^\times(G/N)$ induced by taking $N$-fixed points on $G$-sets. 
Thus $T_V:B^\times(G)\to B^\times(G/N)$ is the map induced by taking $N$-fixed points on $G$-sets. In particular,
if $u\in \partial B^\times (G)$, then by Lemma~\ref{faithEltChar}, $u$ is in kernel of all nontrivial deflation maps. This happens if and only if $|u^K|=1$ for all
subgroups $K$ in $G$ that contain a nontrivial normal subgroup of $G$.
\end{rmk}

\section{groups with abelian subgroups of index $1$ or $2$}

Since $B(G)$ can be embedded into $\bZ^{|\sS_G|}$ and this embedding has finite cokernel, extending scalars by $\bQ$ we have
\[\bQ B(G):=\bQ\otimes_{\bZ} B(G)\cong\bQ^{|\sS_G|},\]
The following theorem gives a formula for
the primitive idempotents of $\bQ B(G)$ in terms of basis elements of $B(G)$. It was originally proved by Gluck (\cite{G}) and independently by Yoshida (\cite{Yo2}).

\begin{theorem}[\cite{B1}, Theorem 2.5]
\label{primidemQBG}
Let $G$ be a finite group. If $H$ is a subgroup of $G$, denote by $e^G_H$ the element of $\bQ B(G)$ defined by
\[e^G_H= \dfrac{1}{|N_G(H)|}\sum_{K\leq H}|K|\mu(K,H)[G/K],\]
where $\mu$ is the M\"obius function of the poset of subgroups of $G$. Then, $e^G_H$=$e^G_K$ if and only if $H=_GK$, and the elements
$e^G_H$ for $H\in \sS_G$ are the primitive idempotents of the $\bQ$-algebra $\bQ B(G)$. Moreover, for $H,K\in \sS_G$ one has $(e^G_H)^K=1$ if $K=H$
and $(e^G_H)^K=0$ if $H\neq K$.
\end{theorem}

Notice that for any element $u\in \bQ B(G)$, one has $ue^G_H=|u^H|e^G_H$, hence $u=\sum_{H\in \sS_G}|u^H|e^G_H$.
For $u\in B^\times(G)$, we set $[F_u]:=\{S\in\sS_G||u^H|=-1\}.$
If $u=[G/G]$, then $[F_u]$ is the empty set. We then can write
\begin{equation}\label{unitBreakdown}
u=[G/G]-\sum_{K\in [F_u]}2e^G_K.
\end{equation}

\begin{lemma}[\cite{B2}, Lemma 6.8]
\label{faithCentBound}
Let $G$ be a finite group with $|Z(G)|>2$, then $\partial B^\times(G)$ is trivial.
\end{lemma}

We immediately obtain the following corollary.

\begin{cor}
\label{trivialFaithAb}
Let $G$ be an abelian group, then $\partial B^\times(G)$ is trivial if $|G|>2$.
\end{cor}

We will make use of Lemma~\ref{faithCentBound} to determine $\partial B^\times(G)$ for groups $G$
with abelian subgroups of index $2$. We begin with a reduction lemma.

\begin{lemma}
\label{C2action}
Suppose $C_2$ acts on an abelian group $G$ by automorphisms.
Let $x$ denote the generator for $C_2$. Suppose $|G^{C_2}|\leq2$. Then $x$ acts
by inversion on the odd part $G_{2'}$ of $G$, and one of the following hold:
\begin{enumerate}
\item[(i)] $G_2$ is cyclic and $x$ acts by inversion on $G_2$.
\item[(ii)] $G_2$ is cyclic of order $2^n\geq8$ and if $g\in G_2$, then $\,^xg=g^{2^{n-1}-1}$.
\item[(iii)] $G_2\cong C_{2^n}\times C_2=\langle a,b| a^{2^n}=b^2=[a,b]=1\rangle$, with $^xb=a^{2^{n-1}}b$ and $^xa=a^eb^\varepsilon$, with 
$e\in\{2^{n}-1,2^{n-1}-1\}$ and $\varepsilon\in\{0,1\}$.
\end{enumerate}
\end{lemma}
\begin{proof}
Let $g\in G$ and set $z=\,^xgg$. Note that $z\in G^{C_2}$. If $g$ has odd order, then so does $z$. Thus $x$ acts by inversion on $g$, otherwise $|G^{C_2}|>2$. If $G_2$ is cyclic of order $2^n$, choose $g$ to be a generator and write $^xg=g^d$, with $1\leq d\leq 2^n$. 
We may assume that $z^2=1$, hence $2^n$ divides $2(d+1)$. This implies that $d=2^{n}-1$ or $d=2^{n-1}-1$, which correspond to cases $(i)$ and $(ii)$ respectively.

Now assume that $G_2$ is not cyclic. If $G_2\cong V_4$ then the statement follows by the constraint that $|G^{C_2}|\leq 2$. Assume then
that $G_2\ncong V_4$. Let $P$ be the subgroup of $G_2$ generated by elements of order $2$. In particular, $P$ is an $\bF_2$ vector space
of dimension equal to the $2$-rank of $G_2$. Let $T_x$ denote the linear transformation on $P$ induced by the action of $x$ on $G$. Then 
$T_x$ has minimal polynomial dividing $(\lambda+1)^2$. Thus, if $P$ has dimension greater than $2$, the eigenspace for $\lambda=1$ has
at least dimension $2$, thus $|G^{C_2}|>2$. So we may assume that 
$G_2= C_{2^n}\times C_{2^m}=\langle a,b \,|\, a^{2^n}=b^{2^m}=[a,b]=1\rangle$, with
$n>1$. Write $^xa=a^eb^\varepsilon$. If $z=\,^xaa$, then we may again assume that $z^2=1$. Thus $e\in\{2^{n}-1,2^{n-1}-1\}$, and 
$b^{2\varepsilon}=1$. Thus $^x(a^2)\in \langle a^2\rangle$ and $a^{2^{n-1}}\in G^{C_2}$. Similarly, we get $b^{2^{m-1}}\in G^{C_2}$. Thus if
$m>1$ or $^xb=b$, we have $|G^{C_2}|>2$. The result follows.
\end{proof}

Recall from the end of Remark~\ref{tensorInductionComp} that an element $u\in B^\times(G)$ is a faithful element if and only if $|u^K|=1$ for any $K\leqslant G$ that contains a nontrivial normal subgroup of $G$. We will use this without citation through the end of the section.

\begin{lemma}
\label{PhiLemma}
Let $G=C_2\ltimes N$ where $N$ is an abelian group of order greater than $2$. Suppose $C_2$ acts on $N$ by inversion and that $N_2$ is cyclic.
Then $\partial B^\times(G)$ has $\bF_2$-dimension $1$. More precisely: If $N_2$ is trivial, then $\partial B^\times(G)$
is generated by the element
\[\Phi_{G}=[G/G]+[G/1]-2[G/I],\]
where $I$ is a representative of the unique conjugacy class of subgroups of order $2$ in $G$.
If $N_2$ is nontrivial, then $\partial B^\times(G)$ is generated by the element
\[\Phi_{G}=[G/G]+[G/1]-[G/I]-[G/J],\]
where $I$ and $J$ are representatives of the two non-central conjugacy classes of subgroups of order $2$.

\end{lemma}
\begin{proof}
Let $x$ be the generator of $C_2$ and let $u\in \partial B^\times(G)$. Then $|u^S|=1$ for any subgroup $S\leqslant G$ which contains a nontrivial normal subgroup. 
Note that any subgroup of $N$ is normal in $G$. Let $S\leqslant G$ with $|u^S|=-1$. Then $S\cap N$ is trivial.
This means that $S$ is a noncentral subgroup of order $2$ or trivial.

Suppose $N_2$ is trivial. By Sylow's theorem, $G$ has a unique conjugacy class of subgroups of order $2$. Let $I=\langle x\rangle$. By 
Theorem~\ref{primidemQBG} and Equation~\ref{unitBreakdown} we may write
\[u=[G/G]+\alpha[G/1]+\beta[G/I],\] 
with $\alpha,\beta\in \bZ$. 
Since $Z(G)$ is trivial, $I$ is self-normalizing. Taking fixed points of $I$ and the trivial subgroup, we get $\beta+1=\pm1$ and $\frac{|G|}{2}\beta+|G|\alpha+1=\pm1$. 
The only integer solutions are $\alpha=\beta=0$ and $\alpha=1$ and $\beta=-2$.

Suppose that $N_2$ is nontrivial and cyclic. Let $y$ denote the generator of $N_2$. Set $I=\langle x\rangle$ and $J=\langle xy\rangle$.
Since $I\not=_GJ$,
we know that $G$ has at least two conjugacy classes of noncentral subgroups of order $2$. Further, since $N_G(I)=Z(G)I$, $N_G(J)=Z(G)J$,
and $|Z(G)|=2$, we have $|N_G(J)|=|N_G(I)|=4$. Thus $[G:N_G(J)]=[G:N_G(I)]=\frac{|N|}{2}$. Since every noncentral subgroup of order $2$ is of the form $\langle xn\rangle$
with $n\in N$, we see that $G$ has exactly two conjugacy classes of noncentral subgroups of order $2$. 
Again by 
Theorem~\ref{primidemQBG} and Equation~\ref{unitBreakdown}, we may write
\[u=[G/G]+\alpha[G/1]+\beta[G/I]+\gamma[G/J],\]
where $\alpha,\beta,\gamma\in \bZ$.
Since $|N_G(I)/I|=|N_G(J)/J|=2$,  taking fixed points of $I,J$ and the trivial subgroup, we have $2\beta+1=\pm1$, $2\gamma+1=\pm1$, 
and $\frac{|G|}{2}\beta+\frac{|G|}{2}\gamma+|G|\alpha+1=\pm1$. The only integer solutions are
$\alpha=\beta=\gamma=0$ and $\alpha=1, \beta=\gamma=-1$ and the result follows.
\end{proof}

\begin{notadef}\label{pseudodihedralDefinition}
We will refer to groups as described in Lemma~\ref{PhiLemma} as \emph{pseudodihedral}.
We will frequently abuse this definition and say a group is pseudodihedral if it is isomorphic to a pseudodihedral group.
If $G$ is trivial, isomorphic to a cyclic group of order $2$, or isomorphic to an pseudodihedral group, 
then $\partial B^\times(G)$ is generated by a unique nontrivial element. We will denote this element as $\Phi_G$.
\end{notadef}

\begin{rmk}\label{PhiCharacterization}
If $G$ is the trivial group, then $\Phi_G=-1$. If $G=C_2$, then $\Phi_G$ can be characterized by its image in the ghost ring, which is 
given by $|\Phi_G^{\{1\}}|=1$ and $|\Phi_G^{C_2}|=-1$.
Similarly, if $G$ is pseudodihedral, then Lemma~\ref{PhiLemma} tells us that $|\Phi_G^X|=-1$ if and only if $X$ is a noncentral subgroup of $G$, of order $2$.
\end{rmk}

It is straightforward to verify that any subgroup or any quotient group of a pseudodihedral group will be either pseudodihedral or abelian.
Equivalently, this means that any subquotient of a pseudodihedral group is either abelian or pseudodihedral. 
This is analogous to a property dihedral groups have. 

Suppose $G=C_2\ltimes N$ is pseudodihedral and $x$ is a generator for $C_2$. With the exception where $G\cong D_8$, the subgroup $N$
is the unique abelian subgroup with index $2$. However, for the case of $D_8$, the subgroup $N$ can be specified as the unique cyclic subgroup
of index $2$. 
Notice that every element of $G-N$ is of the form $xn$ for some $n\in N$. Thus, every element of $G-N$
conjugates elements of $N$ by inversion and every element of $G-N$ has order $2$.

\begin{lemma}\label{quotientSubConnection}
Suppose $G$ is abelian or pseudodihedral. If $S\unlhd G$ is a normal subgroup of $G$, then there is some $X\leqslant G$ such that
$G/S\cong X$. Moreover, for every pseudodihedral subgroup of $Y\leqslant G$, there is a quotient of $G$ isomorphic to $Y$.
\end{lemma}

\begin{proof}
This is trivial if $G$ is abelian, so we consider the case where $G$ is pseudodihedral.
If $S$ is equal to $G$, or trivial, or if $G/S$ has order $2$, then the result is clear. Assume that $S$ is nontrivial and that $|G/S|>2$. 
Since $G$ is pseudodihedral, we let $N$ denote the unique abelian subgroup (in the case $G\cong D_8$, we let $N$ denote the unique cyclic subgroup) of index $2$ in $G$.
We first prove that $S<N$. For the sake of contradiction, suppose that $x\in S$, such that $x\not\in N$.  
Notice that $x$ will have order $2$, conjugation by $x$ inverts elements of $N$, and $S=\langle x,N\cap S\rangle$. 
Because of our assumption that $[G:S]>2$, we can deduce
that $[N:S\cap N]>2$. Thus, there is some $n\in N$ such that $n^2\not\in S\cap N$. The subgroup $S$ is normal in $G$, so we have that $n^{-1}xn=xn^2\in S$ 
and so $x(xn^2)=n^2\in S$, which is a contradiction. So $S$ is contained in $N$.

Since $N$ is abelian, there exists 
$X'\leqslant N$ such that $N/S=\overline{N}\cong X'$, where we use the notation $\overline{g}$ to denote the image of $g\in G$ in the canonical projection of 
$G\to G/S$, and set $X=\langle x, X'\rangle$, where $x$ is any element not in $N$. Write $G/S=\langle\overline{x},\overline{N}\rangle$. We see that $G/S$ is pseudodihedral or isomorphic to
$C_2\times C_2$, since conjugation by 
$\overline{x}$ inverts elements of $\overline{N}$, the subgroup $\overline{N}$ is abelian, the two part $\overline{N}_2$ of $\overline{N}$ is cyclic, and $\langle \overline{x}\rangle\cap\overline{N}$ is trivial. 
Similarly one verifies $X$ is pseudodihedral or 
isomorphic to $C_2\times C_2$. Since $\overline{N}\cong X'$, then $G/S\cong X$.

To prove the last statement, consider that if $Y$ is a pseudodihedral subgroup of $G$, then $Y=\langle xn, Y\cap N\rangle$, for some $n\in N$. 
Since $N$ is abelian, there is some subgroup $M\leqslant N$, such that $Y\cap N\cong N/M$. Moreover, $x$ acts by inversion on $N$, hence
$M\unlhd G$. Thus $G/M\cong Y$. 
\end{proof}

\begin{prop}
\label{abelianSubgroup}
Suppose $G$ is a finite group. If $N<G$ is an abelian subgroup with $[G:N]=2$, then exactly one of the following hold:
$\partial B^\times (G)$ is trivial,
$G$ is isomorphic to $C_2$, or
$G$ is isomorphic to a pseudodihedral group.

\end{prop}

\begin{proof}
If $N$ is trivial or of order $2$, then $G$ is abelian and we are done by Corollary~\ref{trivialFaithAb}. We will now suppose that $|N|>2$.
Let $x\in G$ be an element whose image 
generates $G/N$, then $\langle x, N\rangle=G$. Denote $x^2=z$ and note that $z\in N$, in fact $z\in Z(G)$. We may assume that $z^2=1$, otherwise 
the result holds by Lemma~\ref{faithCentBound}. We may also assume $x\not \in Z(G)$, otherwise the result holds by Corollary~\ref{trivialFaithAb}. 
Since $N$ is abelian, conjugation by $x$ induces an action of $G/N$ on $N$. Any element of $N$ fixed by this action will be in the center of $G$. Therefore, by Lemma~\ref{faithCentBound}, we may assume that we are in one of the three cases described in Lemma~\ref{C2action}. Note that Lemma~\ref{C2action} also allows us to assume
$N_2$ is nontrivial, otherwise $z=1$ and $G=\langle x\rangle \ltimes N$ is pseudodihedral.

Our first goal is to show that if $\partial B^\times(G)$ is nontrivial, then $G$ is isomorphic to $D_8$ or $N$ has a complement in $G$.
Suppose $u\in \partial B^{\times}(G)$ is any nontrivial element. By Equation~\ref{unitBreakdown} we can write

\[u=[G/G]-\sum_{K\in [F_u]}2e^G_K,\]
where $[F_u]=\{S\in \sS_G||u^S|=-1\}$. If $u$ is nontrivial, then $[F_u]$ is nonempty. Recall that for any $S\in [F_u]$, the subgroup $S$ does not contain a nontrivial normal subgroup of $G$. Choose $S$ to be a maximal element of $[F_u]$, then, by 
Theorem~\ref{primidemQBG}, the coefficient in front of $[G/S]$ in $\beta=\sum_{K\in [F_u]}2e^G_K$ is equal to $\frac{2}{[N_G(S):S]}$. Since 
$\beta\in B(G)$, $\frac{2}{[N_G(S):S]}$ is an integer. Thus,
$[N_G(S):S]\leq 2$. If $S$ is trivial, then $N_G(S)=G$ and $[N_G(S):S]>4$, since $N$ is nontrivial. Thus $S$ is not trivial.

If $S\leqslant N$, then $[N:S]\leq[N_G(S):S]\leq2$. The subgroup $S$ cannot be a normal subgroup, since $|u^S|=-1$. Thus $[N:S]=2$.
Since $S$ does not contain a nontrivial normal subgroup of $G$, we may assume that $N_{2'}$ is trivial, so $N$ must be a 2-group.
Consider each of the three cases in Lemma~\ref{C2action} applied here.
For the first two cases, when $N$ is cyclic, the unique subgroup of order two in $N$ is fixed by the action of $G/N$. Any maximal subgroup
of $N$ will contain this subgroup, which contradicts our assumption that $S$ contains no nontrivial normal subgroup of $G$. Now, consider
the third case of Lemma~\ref{C2action}, where $N=\langle a,b| a^{2^n}=b^2=[a,b]=1\rangle$, with $^xb=a^{2^{n-1}}b$ and $^xa=a^eb^\varepsilon$, with 
$e\in\{2^{n}-1,2^{n-1}-1\}$ and $\varepsilon\in\{0,1\}$. Assume for now that $n>1$. Thus $a^{2^{n-1}}\neq 1$ is fixed by the action of $G/N$. If 
$S\neq\langle a\rangle$, then $S\langle a\rangle=N$ implies that $S\cap \langle a\rangle$ is nontrivial since $|N|\geq8$. In either case, 
$\langle a^{2^{n-1}}\rangle< S$ and $S$ contains a nontrivial normal subgroup of $G$.

If $N\cong C_2\times C_2$, then $|G|=8$. Thus $G$ is either abelian or isomorphic to $D_8$, since the quaternion group does not contain a subgroup isomorphic to $C_2\times C_2$. Hence $G\cong D_8$ or $S\not\leqslant N$.

If $G\not\cong D_8$, then we can assume that $S\not\leqslant N$. Thus, there is some element in $S$ of the form $xa$ where $a\in N$. 
Recall that any element of $G-N$ is of the form $xb$ or $x^{-1}b$ for some $b\in N$. Note that since $xN=x^{-1}N$, so conjugation by $x$ and $x^{-1}$ induce the same
action.
Since $N$ is abelian, it follows that, for any $h\in N$, $^{xa}h=\,^{xb}h=\,^{x^{-1}b}h$.
This implies $S$ contains all $G$-conjugates 
of $N\cap S$. In particular, $^x(N\cap S)(N\cap S)\leqslant S$ is a normal subgroup of $G$.
Since $S$ contains no nontrivial normal subgroup of $G$, we may assume that $N\cap S=\{1\}$.
This implies that $|S|=2$ and thus, is a complement to $N$ and $G=S\ltimes N$.

We now prove that if $\partial B^\times (G)$ is nontrivial, then $G=S\ltimes N$ is pseudodihedral. Notice that in the previous paragraph, we showed that any maximal element of 
$[F_u]$ must be a noncentral subgroup of order $2$. 
Let $y$ be a generator of $S$. We only need to consider cases $(ii)$ and $(iii)$ from Lemma~\ref{C2action}, otherwise $G$ is pseudodihedral.
Note that $N_G(S)=C_G(y)S=Z(G)S$. 
This establishes that $S$ is contained in a conjugacy class of size $\frac{|G|}{4}$. Each noncentral subgroup of order $2$ is generated by an element of the form $ym\in G$ for some unique $m\in N$. 
Since each conjugacy class of noncentral subgroups of order $2$ must also have size $\frac{|G|}{4}=\frac{|N|}{2}$, there are at most two such classes. 
There will only be
two conjugacy classes of noncentral subgroups of order $2$
if every element of the form $ym\in G$, where $m\in N$ has order $2$.  It is easy to verify that this is not so for
each of the cases $(ii)$ and $(iii)$ from Lemma~\ref{C2action}.
Thus by Theorem~\ref{primidemQBG} we can write
\[u=[G/G]+\alpha[G/1]+\beta[G/S],\]
for some integers $\alpha$ and $\beta$.
Considering fixed points of $S$ and and the trivial group produces equations of the form $2\beta+1=\pm1$ and $\frac{|G|}{2}\beta+|G|\alpha+1=\pm1$. Because
$G$ is not abelian, $|G|>4$ and the only integer solutions are $\alpha=\beta=0$. The result follows.
\end{proof}

\begin{theorem}
\label{basisofab2}
Let $G$ be a finite group with an abelian subgroup of index $1$ or $2$. Let $\mathcal{N}$ denote the set of normal subgroups
of $G$ such that $G/N$ is trivial, isomorphic to $C_2$, or isomorphic to a pseudodihedral group. Then $\{\Inf_{G/N}^G(\Phi_{G/N})\}_{N\in \mathcal{N}}$ is a basis for 
$B^\times(G)$.
\end{theorem}

\begin{proof}
Note that factor groups of $G$ will also have abelian subgroups of index $1$ or $2$. If $N$ is normal in $G$, then Corollary~\ref{trivialFaithAb}, Proposition~\ref{PhiLemma}, and Proposition~\ref{abelianSubgroup} altogether imply that $\partial B^\times(G/N)$ is either trivial or generated by $\Phi_{G/N}$. We can apply Proposition~\ref{normChar} to $G$ and the result follows.
\end{proof}

We remark that this generalizes and unifies some of Matsuda's results on the order of the unit group of the Burnside ring for abelian groups and dihedral groups (\cite{Ma}, Examples $4.5$ and $4.8$). We can extend this result to a larger class of groups by the following proposition due to Bouc. We remark that the full proof of this result requires the Feit-Thompson Odd Order Theorem.

\begin{prop}[\cite{B2}, Proposition 6.5]\label{oddIndexIso}
Let $G$ be a finite group, and $N$ be a normal subgroup of odd index in $G$. Then the group $G/N$ acts on $B^\times(N)$ and the maps 
$\Res^G_{N}$ and $\Ten_{N}^G$ induce mutual inverse isomorphisms between $B^\times(G)$ and $B^\times(N)^{G/N}$.
\end{prop}

\begin{cor}\label{oddIndexCor}
Let $G$ be a finite group and $N$ a normal subgroup of odd index in $G$. 
Suppose $N$ has an abelian subgroup of index at most $2$. If $\mathcal{N}$ and
$\mathcal{L}=\{\Inf_{N/K}^N(\Phi_{N/K})\}_{K\in \mathcal{N}}$ are respectively the indexing set and basis of $B^\times(N)$ described in Theorem~\ref{basisofab2} 
applied to $N$, then $G/N$ acts on $\mathcal{L}$ and, denoting $L_1,\dots, L_k$ to be the orbit sums of this action, $\{\Ten^G_{N}L_i\}_{i=1}^k$
is a basis of $B^\times(G)$.
\end{cor}

\begin{proof}
$G$ acts on $N$ by conjugation, hence on $B(N)$ by ring automorphism, which restricts to an action on $B^\times(N)$ by group automorphism. The action of
$N$ is trivial, so we obtain an action of $G/N$ on $B^\times(N)$. By Proposition~\ref{oddIndexIso} it suffices to show that $\mathcal{L}$ is 
invariant under this action.
Choose $g\in G$ and let $c_g$ denote the automorphism of $N$ given by conjugation by $g$. Let 
$u\in B^\times(N)$, then $gN\cdot u=\Iso(c_g)(u)$. Hence, if $u=\Inf_{N/K}^N(\Phi_{N/K})$ for some $K\in \mathcal{N}$, then the action of
$gN$ on $u$ is equal to
\[\Iso(c_g)\circ\Inf_{N/K}^N(\Phi_{N/K})=\Inf^N_{N/^gK}\circ\Iso(\overline{c}_g)(\Phi_{N/K}),\]
where $\overline{c}_g$ is the isomorphism $N/K\to N/^gK$, induced by $c_g$. Since $\overline{c}_g$ is an isomorphism, $^gK\in \mathcal{N}$. Furthermore,
$\Iso(\overline{c}_g)(\Phi_{N/K})$ will be in the kernels of all nontrivial deflation maps, so it must be a faithful element of $B^\times(N/^gK)$, hence
$\Iso(\overline{c}_g)(\Phi_{N/K})=\Phi_{N/^gK}$. Thus $\mathcal{L}$ is invariant under the action of $G/N$.
\end{proof}

\begin{rmk}
The previous corollary is a generalization of Theorem~\ref{basisofab2}, though in general its application will involve a choice of a normal subgroup $N$ of odd index. 
The statement of the corollary implies that any choice of a normal subgroup $N$ satisfying the hypothesis will produce a basis for $B^\times(G)$, 
and different choices might result in 
different bases. However, we note that the intersection of such normal subgroups also satisfies the hypothesis, thus
we are able to designate a smallest such choice as the standard if needed.
\end{rmk}

\begin{rmknota}\label{specialSubcat}
It can be easily verified that the full subcategory of $\C$ whose objects are groups with normal subgroups $N$ of odd index, 
where $N$ has an abelian subgroup of index at most $2$, is a replete subcategory of $\C$. We will denote this subcategory by $\C'$.
\end{rmknota}

\begin{example}\label{unitGroupComputations}
Suppose $G$ is a finite group and $N\cong D_{2n}$ is a normal subgroup of 
$G$ with odd index, where $D_{2n}$ denotes the dihedral group of order $2n$. Then $N$ has
a unique cyclic subgroup $Z$, of index $2$. Let $d(n)$ denote the number of positive divisors of $n$.
Notice that if $K\unlhd N$, then $N/K$ is dihedral, cyclic of order $2$, or trivial, if and only if $|N/K|\neq 4$. 
Further, if $n$ is odd, every proper normal subgroup of $G$ is a subgroup of $Z$, and this corresponds exactly to the divisors of $n$. In the case where $n=2k$ for some integer $k>1$, we have three normal subgroups of index $2$, namely $Z$ and two that are isomorphic to $D_{2k}$. Every other proper normal subgroup of $N$ is a proper subgroup of $Z$ and exactly $d(n)-2$ of these have index other than $4$. By Theorem~\ref{basisofab2}, $|B^\times(N)|$ is equal to $2^{d(n)+1}$ if $n$ is odd and $2^{d(n)+2}$ if $n$ is even.

If $n$ is odd, every normal subgroup of $N$ is characteristic. In the case where $n=2k$, every normal subgroup is characteristic
except for the two normal subgroups isomorphic to $D_{2k}$. However, since $G/N$ has odd order,
it stabilizes these two subgroups. Thus, by
Corollary~\ref{oddIndexCor} $\Ten_N^G:B^\times(N)\to B^\times(G)$ is an isomorphism. 
\end{example}

We end this section with a pair of lemmas which will be helpful in the later sections.

\begin{lemma}\label{abelianGeneratedBy-1}
Let $G$ be an abelian group. Then
\[B^\times(G)=B(G,\{1\})\Phi_{\{1\}}=\langle \Ten_H^G(-1)\rangle_{H\leqslant G}.\]
\end{lemma}

\begin{proof}
This can be seen by straightforward calculation when $G$ is trivial or $C_2$. The general case follows from Theorem~\ref{basisofab2} and the
transitivity of induction.
\end{proof}

\begin{lemma}
\label{PhiRestriction}
Let $G$ be a pseudodihedral group and $H\leqslant G$. If $H$ is trivial or a noncentral group of order $2$, then $\Res^G_H(\Phi_G)=-\Phi_H$. If
$H$ is isomorphic to a pseudodihedral group then $\Res^G_H(\Phi_G)=\Phi_H$. Moreover, for any subquotient, $S$, of $G$ we have
\[B(S,G)\Phi_G=B^\times(S).\]
\end{lemma}
\begin{proof}
We may write $G=\langle x, N\rangle$, where $N$ is abelian and $x$ acts by inversion on $N$.
If $H$ is trivial or a noncentral group of order $2$, it is straightforward to verify that $\Res^G_H(\Phi_G)=-\Phi_H$. 
If $H$ is pseudodihedral, then $H=\langle xy, N\cap H\rangle$
for some $y\in N$. 
Notice that for any $S\leqslant H$, we have $|\Res^G_H(\Phi_G)^S|=-1$ if and only if $S$ is a noncentral subgroup of $H$, of order $2$ for some. 
Thus $|\Res^G_H(\Phi_G)^S|=|\Phi_H^S|$ for any $S\leqslant H$, by Remark~\ref{PhiCharacterization}. It follows that $\Res^G_H(\Phi_G)=\Phi_H$.

To prove the last statement, notice that $S$ is either abelian or pseudodihedral.
If $S$ is abelian, this follows easily from Lemma~\ref{abelianGeneratedBy-1} as long as we can prove it for $S=\{1\}$. For this, it is sufficient
to show that $-1=\Phi_{\{1\}}\in B(\{1\},G)\Phi_G=B^\times(\{1\})$. For this, it is straightforward to verify that 
\[\Def^{\langle x\rangle}_{\langle x\rangle/\langle x\rangle}\circ\Res^G_{\langle x\rangle}(\Phi_G)=\Phi_{\langle x\rangle/\langle x\rangle}.\]

In the case when $S$ is pseudodihedral, we first show the lemma holds in the case where $S=G$.
Lemma~\ref{quotientSubConnection} tells us that each factor group of $G$ is isomorphic to a subgroup of 
$G$. If $N$ is a normal subgroup of $G$ and $\partial B^\times(G/N)$ is nontrivial, then let $T\leqslant G$ be isomorphic to $G/N$
and $f:T\to G/N$ is an isomorphism. If $G/N$ has order $2$, then choose $T$ to be noncentral. Then $\Iso(f)\circ\Res^G_T(\Phi_G)=\epsilon*\Phi_{G/N}\in B^\times(G/N)$ (here $``*"$ is used to denote the group law of $B^\times(G/N)$), 
where $\epsilon=-1\in B^\times(G/N)$ if $T$ is trivial or cyclic of order $2$ and $\epsilon=1\in B^\times(G/N)$ otherwise. Thus 
\[\Inf_{G/N}(\Phi_{G/N})\in B(G,G)\Phi_G\]
 and by Theorem~\ref{basisofab2}, $B(G,G)\Phi_G=B^\times(G)$.

If $S$ is any pseudodihedral subquotient of $G$, then $S=X/M$ for some $X\leqslant G$ and $M\unlhd X$. 
By the first part of the lemma and Lemma~\ref{quotientSubConnection}, we have 
\[B(X/M,X/M)\Phi_{X/M}\subseteq B(X/M,G)\Phi_G\] 
and the result follows by replacing $G$ in the second paragraph with $X/M$.
\end{proof}

\section{Residual groups for $B^\times$}

Suppose $R$ is a commutative ring with identity and $\D$ is a replete subcategory of $\C$.

\begin{notadef}
Let $F$ be a biset functor on $\D$ with values in $_R\text{Mod}$. For any object $G$ of $\D$, we set $F^<(G)$, to be the $E_G$-submodule
\[F^<(G):=\sum_{|H|<|G|}\Hom_{R\D}(H,G)F(H)\subseteq F(G),\]
where $H$ runs through objects of $\D$ which have smaller cardinality than $G$. If $G$ is an object of $\D$, such that 
$F^<(G)\subsetneq F(G)$, then we say that $G$ is \emph{residual with respect to $F$}.\\

It follows by Remark~\ref{elementaryBisetPresentation} that 
\[F^<(G)=\sum_{H}\Hom_{R\D}(H,G)F(H)\]
where $H$ runs over subquotients of $G$.
Thus, if $\D$ and $\D'$ are replete subcategories of $\C$ such that $\D\subset \D'$, then the residual objects of $\D$ are just the residual objects of $\D'$ that are in $\D$.
 
We denote by $\mathcal{R}_{\D,F}$ a set of representatives, up to isomorphism, of objects in $\D$ which are residual with respect to $F$
and call $\mathcal{R}_{\D,F}$ a \emph{complete set of residuals} for $F$ in $\D$.
\end{notadef}
In the next proposition, we use the notation $``H\prec G"$, which means $H$ is isomorphic to a subquotient of $G$.

\begin{prop}\label{completeSetRes}
Suppose $\D$ is a replete subcategory of $\C$ and $F$ is a biset functor on $R\D$. Let $\mathcal{R}_{\D,F}$ and $\mathcal{R}'_{\D,F}$ be complete sets of residuals for $F$ in $\D$. If $G\in \mathcal{R}_{\D,F}$ then $\mathcal{R}'_{\D,F}$ has an object $G'$ which is isomorphic to $G$. Furthermore, for
any object $G\in \D$,
\[F(G)=\sum_{\underset{H\prec G}{H\in \mathcal{R}_{\D,F}}}\Hom_{R\D}(H,G)F(H).\]
\end{prop}
\begin{proof}
If $G$ is a residual object of $\D$ with respect to $F$, then any group isomorphic to $G$ will also be residual with respect to $F$. The second statement follows from easy induction on the order of $G$ and the definition of residual with respect to $F$.
\end{proof}

Suppose $\D\subset \D'$. If $\mathcal{R}_{\D,F}$ is a complete set of residuals for $F$ in $\D$, then there is a complete set of residuals $\mathcal{R}_{\D',F}$ for $F$ in $\D'$
such that 
$\mathcal{R}_{\D,F}\subset \mathcal{R}_{\D',F}$. Also, for any object $G$ of $\D$, we always have that
\[\sum_{1\neq N\unlhd G}\text{Im}F(\Inf^{G}_{G/N})\subseteq I_GF(G)\subseteq F^{<}(G)\subseteq F(G),\]
as $R$-modules.

We now consider the biset functor $B^\times\in \F_{\C'}$ where $\C'$ is the subcategory of $\C$ from Remark~\ref{specialSubcat}.
To be consistent with the usual biset notation, for any finite groups $G$, we have opted to denote the group law of $B^\times(G)$ additively throughout the rest of the paper.

\begin{rmk}\label{codimension1}
Here we recall a particularly useful consequence of Lemma~\ref{PhiLemma}. Suppose $G$ is pseudodihedral. Then
\[B^\times(G)/(\sum_{1\neq N\unlhd G}\mathrm{Im}\Inf^G_{G/N})\]
is generated by the image of $\Phi_G$, as an $\bF_2$-space. In particular, $B^\times(G)/(\sum_{1\neq N\unlhd G}\mathrm{Im}\Inf^G_{G/N})$ has $\bF_2$-dimension $1$. 
\end{rmk}

\begin{prop}\label{residImpliesQuasiD}
Suppose $G\in\C'$.
\begin{enumerate} 
\item[(a)] If $G$ is residual with respect to $B^\times$ then $G$ is trivial or pseudodihedral. 
\item[(b)] If $G$ 
is pseudodihedral, then $G$ is residual with respect to $B^\times$ if and only if 
\[(B^\times)^<(G)=\sum_{1\neq N\unlhd G}\mathrm{Im}\Inf^G_{G/N}.\]
\end{enumerate}
\end{prop}

\begin{proof}
By Theorem~\ref{abelianSubgroup} and Corollary~\ref{oddIndexCor}, $G$ must be trivial, isomorphic to $C_2$,
or pseudodihedral group. But by Lemma~\ref{abelianGeneratedBy-1}, $C_2$ is not residual.

For any finite group $G$, we always have
\[\sum_{1\neq N\unlhd G}\text{Im}\Inf^G_{G/N}\subseteq (B^\times)^<(G) \subseteq B^\times(G).\]
When $G$ is pseudodihedral, it follows by Remark~\ref{codimension1} that $G$ is residual if and only if $(B^\times)^<(G)=\sum_{1\neq N\unlhd G}\text{Im}\Inf^G_{G/N}$.
\end{proof}

\begin{rmk}\label{quasiMax}
Suppose $G$ is a pseudodihedral group and let $N$ denote its unique abelian subgroup (the unique cyclic subgroup when $G\cong D_8$) of index $2$. 
If $H$ is any maximal subgroup of $G$ other than $N$, then $H=\langle x,N\cap H\rangle$, where $x\in G-N$ and $N\cap H$
is a maximal subgroup of $N$. Further, conjugation by $x$ inverts elements of $N\cap H$. So $H$ is pseudodihedral as long as $|N\cap H|>2$. 
Notice that $N\cap H$ is a maximal
subgroup of $N$. Therefore,
all maximal subgroups of $G$, other than $N$, are also pseudodihedral, as long as $\frac{|G|}{2}=|N|$ is not a prime or twice a prime.
\end{rmk}

\begin{prop}\label{tensoringReduction}
Let $G$ be a pseudodihedral group. Let $N<G$ denote the unique abelian subgroup (the unique cyclic subgroup when $G\cong D_8$) with index $2$. 
If $G$ is not residual with respect to $B^\times$ and $|N|$ is not prime or twice a prime, then there is some pseudodihedral maximal subgroup $H<G$,
for which $\Ten_H^G(\Phi_H)$ and $\Phi_G$ have the same image in 
$B^\times(G)/(\sum_{1\neq K\unlhd G}\mathrm{Im}\Inf^G_{G/K})$.
\end{prop}

\begin{proof}
Since $G$ is not residual, it follows from Remark~\ref{codimension1} that there is some proper subgroup
$H<G$ and an element $u\in B^\times(H)$, such that $\Ten^G_H(u)$ and $\Phi_G$ have the same image in
\[B^\times(G)/ (\sum_{1\neq K\unlhd G}\mathrm{Im}\Inf^G_{G/K}).\]
By the transitivity of induction, we can assume that $H$ is a maximal subgroup of $G$. If $H=N$, then 
$\partial B^\times(N)$ is trivial by Corollary~\ref{trivialFaithAb}, since $|N|>2$.
Thus by Lemma~\ref{faithEltChar} we can write
\[u=\sum_{i=1}^k\Inf^N_{N/N_i}(u_i),\]
where $N_1,\dots, N_k$ are nontrivial normal subgroup of $N$ and $u_i\in B^\times(N/N_i)$, for each $i=1,\dots, k$.
Each subgroup of $N$ is normal in $G$, since $G$ is pseudodihedral. So by Remark~\ref{ebisetTranCom}
\[\Ten^G_H(u)=\Ten^G_H(\sum_{i=1}^k\Inf^N_{N/N_i}(u_i))=\sum_{i=1}^k\Inf^G_{G/N_i}\Ten^{G/N_i}_{N/N_i}(u_i).\]
Hence $\Ten_H^G(u)$ is trivial in $B^\times(G)/ (\sum_{1\neq K\unlhd G}\Im\Inf^G_{G/K})$, a contradiction. We can now assume $H\neq N$ and so it
is pseudodihedral by assumption (see Remark~\ref{quasiMax}).

We want to show that $\Ten_H^G(u)$ and $\Ten_H^G(\Phi_H)$ have the same image 
modulo $\sum_{1\neq K\unlhd G}\mathrm{Im}\Inf^G_{G/K}$. To this end, we first show that $u\not\in \sum_{1\neq S\unlhd H}\Im\Inf^H_{H/S}$.
For the sake of contradiction, suppose $u\in \sum_{1\neq S\unlhd H}\Im\Inf^H_{H/S}$. If $S$ is any nontrivial normal subgroup of $H$, 
then $S\cap N$ cannot be trivial. Otherwise, $S$ is a non-central subgroup of order two and the only elements that normalize $S$ in $G$, other than $S$, will be central in $G$.
This would imply $|H|\leq4$ which contradicts $H$ being pseudodihedral. So, by the 
transitivity of inflation
\[u\in \sum_{1\neq S\unlhd H}\Im\Inf^H_{H/(S\cap N)},\]
where each $S\cap N$ is a nontrivial normal subgroup of $G$.
So again by Remark~\ref{ebisetTranCom}, we have
\[\Ten^G_H(u)\in\sum_{1\neq K\unlhd G}\Im\Inf^G_{G/K},\]
which contradicts that $\Ten_H^G(u)$ and $\Phi_G$ have the same image 
modulo $\sum_{1\neq K\unlhd G}\mathrm{Im}\Inf^G_{G/K}$.

This allows us to assume that $u=\Phi_H+u'$ where $u'\in \sum_{1\neq S\unlhd H}\Im\Inf^H_{H/S}$.
However, as above, we have
\[\Ten^G_H(u')\in\sum_{1\neq K\unlhd G}\Im\Inf^G_{G/K}.\]
This implies that $\Ten_H^G(u)$ and $\Ten_H^G(\Phi_H)$ have the same image 
modulo $\sum_{1\neq K\unlhd G}\mathrm{Im}\Inf^G_{G/K}$ and the result follows.
\end{proof}

Together with Proposition~\ref{residImpliesQuasiD}, the next proposition determines which objects in $\C'$ are residual with respect to $B^\times$. 
It will be used in the next section when we characterize the lattice of subfunctors of $B^\times$ over $\C'$ and find its composition factors.

\begin{rmk}\label{D8}
Before we prove the following proposition, we remark that $D_8$ is not residual. This follows from the unnecessarily strong machinery given by Theorem $8.5$ in \cite{B2}.
It can also be proved easily from straightforward calculations similar to the case in Proposition~\ref{residualGroupsAbIndex2}, where we
consider groups isomorphic to $D_{2p}$, with $p$ an odd prime. We consider this case settled and do not include it in the next proposition to simplify the characterization.
\end{rmk}

\begin{prop}\label{residualGroupsAbIndex2}
Suppose $G$ is a pseudodihedral group not isomorphic to $D_8$. Let $N$ denote the unique abelian subgroup of $G$ with index $2$.
\begin{enumerate}
\item[(A)] When $N\cong C_p$, the group $G$ is not residual with respect to $B^\times$ if and only if $p\equiv 3\, (\text{mod} \,\,4)$.
\item[(B)] When $N$ is not simple, the group $G$ is not residual with respect to $B^\times$ if and only if $|N_2|=2$ or $N$ has an element of order $p^2$ for some odd prime $p$.
\end{enumerate}
\end{prop}
\begin{proof}
We frequently make use of (\ref{tenInduct}) from Remark~\ref{tensorInductionComp}. We write it again, for 
easy reference. If $H\leqslant G$ and $a\in B^\times(H)$ then
\begin{equation}\label{retensor}
\Ten^G_H(a)=\left(\prod_{g\in[L\backslash G/H]}|a^{(g^{-1}Lg\cap H)}|\right)_{L\in \sS_G},
\end{equation}
where $\sS_G$ is a set of representatives of conjugacy classes of subgroups of $G$.
Write $G=C_2\ltimes N$ and denote the generator of $C_2$ by $x$.
Since $G$ is pseudodihedral recall that $N_2$ is cyclic and $x$ acts on $N$ by inversion. 
We split the rest of the proof into the two cases based on whether or not $N$ is simple.\\

\noindent{\bf Part (A):}\\\\
We first consider the case where $N=C_p$ for some prime $p$. Then $p\neq2$ since $G$ is pseudodihedral. By Theorem~\ref{basisofab2},
$|B^\times(G)|=2^3$. By Lemma~\ref{abelianGeneratedBy-1} and Lemma~\ref{quotientSubConnection}, for any abelian group $L$, we have
\[B^\times(L)=B(L,\{1\})\Phi_{\{1\}}=\langle\Ten_H^L(-1)\rangle_{H\leqslant L}.\]
Since each proper subgroup of $G$ is abelian, and induction is transitive, the group $G$ is not residual if and only if
\[B^\times(G)=B(G,\{1\})\Phi_{\{1\}}=\langle\Ten_H^G(-1)\rangle_{H\leqslant G}.\]
$G$ has $4$ conjugacy classes of subgroups. Denote $\sS_G=\{\{1\}, \langle x\rangle, C_p,G\}$ a set of representatives of these classes. By 
$(3)$ (ordering $\sS_G$ by increasing cardinality) we compute
\[\Ten_{\{-1\}}^G(-1)=\begin{pmatrix}1\\-1\\1\\-1\end{pmatrix},\quad\Ten_{\langle x\rangle}^G(-1)=\begin{pmatrix}-1\\(-1)^r\\-1\\-1\end{pmatrix}\]
\[\Ten_{C_p}^G(-1)=\begin{pmatrix}1\\-1\\1\\-1\end{pmatrix},\quad\Ten_G^G(-1)=\begin{pmatrix}-1\\-1\\-1\\-1\end{pmatrix}\]
where $r$ is equal to the number of double cosets in $\langle x\rangle\backslash G /\langle x\rangle$. Explicitly, by $(3)$
\[|\Ten_{\langle x\rangle}^G(-1)^{\langle x\rangle}|=\prod_{g\in[\langle x\rangle\backslash G/\langle x\rangle]}(-1)=(-1)^{|\langle x \rangle \backslash G/\langle x\rangle|}\]
Thus $|B(G,\{1\})(-1)|=|B^\times(G)|$
if and only if $r=|\langle x \rangle \backslash G/\langle x\rangle|=1+\frac{p-1}{2}$ is even. This happens if and only if $p\equiv 3\, (\text{mod} \,\,4)$. Hence, $G$ is 
not residual with respect to $B^\times$ if and only if $p\equiv 3\, (\text{mod} \,\,4)$.\\

\noindent{\bf Part (B):}\\\\
We now consider the case where $N$ is not simple. When $|N_2|= 2$ or $N$ contains an element of order $p^2$ for some odd prime $p$, we will establish that there
is some proper subgroup $H<G$ and some $w\in B^\times(H)$, such that $\Ten_H^G(w)$ and $\Phi_G$ have the same image in 
$B^\times(G)/\sum_{1\neq K\unlhd G}\Im\Inf_{G/K}^G$. On the other hand, if
$|N_2|\neq2$ and $N$ has no element of order $p^2$, for any odd prime $p$,
then Remark~\ref{quasiMax} and Proposition~\ref{tensoringReduction} reduces proving $G$ is residual
to showing $\Ten_H^G(\Phi_H)$ is contained in $\sum_{1\neq K\unlhd G}\Im\Inf_{G/K}^G$, for all pseudodihedral maximal subgroups $H$.

For the rest of the proof, we let $z$ denote a generator of $N_2$ and set $I=\langle x\rangle$ and $J=\langle zx\rangle$. Recall that
if $N_2$ is nontrivial, then $I$ and $J$ are representatives of the two conjugacy classes of noncentral subgroups of order two in $G$.
In this case, the group $N$ has a unique maximal subgroup of index $2$, which we denote as $N'$. 
Further, the group $G$ has exactly two non-abelian maximal subgroups of index $2$. 
These are $H_I:=\langle x,N'\rangle$ and $H_J:=\langle xz,N'\rangle$. Note that since $G\not\cong D_8$, both $H_I$ and $H_J$ are pseudodihedral.
(Also note that in the case where $N_2$ is trivial, $I=J$, and $I$ is a representative of the unique conjugacy class of noncentral subgroups of order two in $G$)\\

\noindent{\bf Claim (i):} \emph{If $N$ has an element of order $p^2$ for some odd prime $p$, then $G$ is not residual with respect to $B^\times$.}\\

Fix $y\in N_p$ to have maximal order, $p^n$. Let $C$ be a complement of $\langle y\rangle $ in $N$. Set
$H'=\langle y^p\rangle\times C$ and $H=\langle x,H'\rangle$. By construction $H'$ has no complement in $N$. 
Since $H'$ is a maximal subgroup of $N$ we have that $H$ is a maximal subgroup of $G$. Fix a subgroup $X\leqslant G$. We need to establish
that $|\Ten_H^G(\Phi_H)^X|=|\Phi_G^X|$. 

Note again that $H'$ has no complement in $N$. Furthermore, any nontrivial subgroup $S\leqslant N$
will intersect with $H'$, hence $H$, nontrivially. Thus, for any $g\in G$, if $^gX\cap H$ is a noncentral subgroup of order $2$ in $H$, then
$X$ is already a noncentral subgroup of order $2$ of $G$. So by $(3)$, we have
\[|\Ten_H^G(\Phi_H)^X|=1\]
unless $X$ is a noncentral subgroup of order $2$. 

Since $H$ has index $p$ in $G$ and $p\neq 2$,
then both $I$ and $J$ are subgroups of $H$. Suppose $X=I$. Then $g^{-1}Xg\cap H\neq \{1\}$ if and only if $g^{-1}xg=xg^2\in H$. Because $H$ has odd index in $G$, the element
$xg^2\in H$ if and only if $g\in H$. Thus $g^{-1}Xg\cap H\neq \{1\}$ if and only
if $XgH=XH$ as double cosets in $X\backslash G/H$. Again by $(3)$,
\[|\Ten_H^G(\Phi_H)^X|=|\Phi_H^X|=-1.\]
When $X=J$, the argument is similar and we get
\[|\Ten_H^G(\Phi_H)^X|=|\Phi_H^X|=-1.\]
Thus, $|\Ten_H^G(\Phi_H)^X|=-1$ if and only if $X\leqslant G$ is a noncentral subgroup of order $2$.
By Remark~\ref{PhiCharacterization}, $|\Ten_H^G(\Phi_H)^X|=|\Phi_G^X|$ for any $X\leqslant G$. Thus $\Ten_H^G(\Phi_H)=\Phi_G$. Therefore $G$ is not residual.\\

\noindent{\bf Claim (ii):} \emph{If $|N_2|=2$, then $G$ is not residual with respect to $B^\times$.}\\

Since $N_2$ is not trivial $H_I$ and $H_J$ are defined and we set $v=\Ten_{H_I}^G(\Phi_{H_I})$. Note also that $N_{2'}$ is nontrivial, since $|N_2|=2$ and $|N|>2$ by $G$ being pseudodihedral. 
Fix a subgroup $X$ of $G$. We wish to compute $|v^X|$. 
If $X$ is trivial then it is clear by (\ref{retensor}) that $|v^X|=1$. Suppose $X$ is nontrivial and
$X\cap N_{2'}\neq\{1\}$. Since $X\cap N_{2'}=\,^g(X\cap N_{2'})\leqslant\,^gX\cap H$ for any $g\in G$, 
it follows that $^gX\cap H$ is not a noncentral subgroup of order $2$ of $H$, for any $g\in G$. So $|v^X|=1$, by $(\ref{retensor})$.

Now suppose $X$ is nontrivial and $X\cap N_{2'}=\{1\}$. Then $X$ is conjugate to $N_2$, $I$, $J$, or $JN_2=IN_2$. 
We apply (\ref{retensor}) to each of these subgroups. For $N_2$ it is easy to verify that
\[|v^{N_2}|=|\Phi_{H_I}^{(N_2\cap H_I)}|=|\Phi_{H_I}^{\{1\}}|=1.\]
For $I$, notice that for any $g\in G$, $g^{-1}Ig\cap H_I=g^{-1}Ig$, since $H_I$ is normal in $G$ and $I< H_I$.
Since $|I\backslash G/H_I|=|G/H_I|=2$, we have that $|v^I|=1$.
Considering $J$, notice that $g^{-1}Jg\cap H_I=\{1\}$ for any $g\in G$, since $J\not< H_I$, thus $|v^J|=1$.
Lastly, for $IN_2$, notice that $|IN_2\backslash G/H_I|=1$. Hence,
\[|v^{IN_2}|=|\Phi_{H_I}^{IN_2\cap H_I}|=|\Phi_{H_I}^{I}|=-1.\]
Therefore, $v$ is the unique element of
$B^\times(G)$ characterized by $|v^X|=-1$ if and only if $X\leqslant G$ is in the same conjugacy class as $IN_2$.

We now consider the element $u=\Inf_{G/N_2}^G(\Phi_{G/N_2})$ and show that $\Phi_G=u+v\in B^\times(G)$ (noting that $G/N_2$ is pseudodihedral, thus $\Phi_{G/N_2}$ is defined.) Recall
that inflation between unit groups of Burnside rings comes from restricting the inflation ring morphism between Burnside rings. 
Thus, for any $X\leqslant G$, we have $|u^X|=|\Phi_{G/N_2}^{(XN_2/N_2)}|=-1$ if and only if 
the image of $X$ through the canonical projection $G\to G/N_2$ is a noncentral subgroup of order $2$. 
In other words, $|u^X|=-1$ if and only if $X$ is conjugate to $I$, $J$, or $IN_2$. 
This implies that $v+u\in B^\times(G)$ is the unique element characterized by $|(u+v)^X|=-1$ if and only if $X$ is a noncentral subgroup of order $2$. 
It follows that $u+v=\Phi_{G}$, by Remark~\ref{PhiCharacterization}. Therefore $G$ is not residual with respect to $B^\times$.
(We note that $u+v=\Phi_{G}$ would still be true if we had set $v=\Ten_{H_J}^G(\Phi_{H_J})$.)\\

\noindent{\bf Claim (iii):} \emph{If $p$ is an odd prime that divides $N_{2'}$ and $N$ has no element of order $p^2$, 
then $\Ten_H^G(\Phi_H)$ is contained in $\sum_{1\neq K\unlhd G}\Im\Inf_{G/K}^G$, for any subgroup $H<G$ of index $p$.}\\

Suppose $H\leqslant G$ has index $p$. Denote $H'=H\cap N$, then $H'$ is a maximal subgroup of $N$ with index $p$. Furthermore, there is some $n\in N_p$ 
such that $H=\langle xn, H'\rangle$. So for $h=n^{\frac{1-p}{2}}$, we have $\,^hH=\langle x, H'\rangle$, since $^h(xn)=xn^{1+p-1}=x$. If $c_h:G\to G$ denotes the automorphism induced by conjugation with $h$, then
\[\Iso(c_h)\circ \Ten_H^G(\Phi_H)=\Ten_{c_h(H)}^G(\Iso(c_h)(\Phi_H))=\Ten_{c_h(H)}^G(\Phi_{c_h(H)}),\]
since the element $\Iso(c_h)(\Phi_H)\in B^\times(c_h(H))$ is clearly a nontrivial faithful element and thus $\Iso(c_h)(\Phi_H)=\Phi_{c_h(H)}$. It follows that
$\Ten_{H}^G(\Phi_{H})$ is in the image of $\sum_{1\neq K\unlhd G}\Im\Inf_{G/K}^G$ if and only if $\Ten_{c_h(H)}^G(\Phi_{c_h(H)})$ is in the image of
$\sum_{1\neq K\unlhd G}\Im\Inf_{G/K}^G$. So it is sufficient for us to assume $H=\langle x, H'\rangle$ and show that 
$\Ten_{H}^G(\Phi_{H})$ is in the image of $\sum_{1\neq K\unlhd G}\Im\Inf_{G/K}^G$.

Write $|N_p|=p^m$. Our assumption implies that $N_p$ is an elementary abelian $p$-group of rank $m$. 
Thus, every maximal subgroup of $N_p$ has 
$\frac{(p^m-1)-(p^{m-1}-1)}{p-1}=p^{m-1}$ complements in $N_p$. 
It follows that every maximal subgroup of $N$
with index $p$ has $p^{m-1}$ complements in $N$. Let $Q_1,\cdots, Q_{p^{m-1}}$ denote the complements of $H'$ in $N$. 

We again use (\ref{retensor}) to compute $\Ten_{H}^G(\Phi_{H})$. Let $X\leqslant G$.
If $X\in\{\{1\}, Q_1,\dots, Q_{p^{m-1}}\}$, then it is easy to compute that $|\Ten_{H}^G(\Phi_{H})^X|=1$.
If $X\cap H'$ is nontrivial, then $^g(X\cap H')\leqslant\ \,^gX\cap H$ is not a noncentral subgroup of order $2$ of $H$ for any $g\in G$.
Thus again, $|\Ten_H^G(\Phi_H)^X|=1$. 
The for the remaining cases, we may assume $X\cap H'=\{1\}$ and $X\not\leqslant N$. It is straightforward to verify this implies $X$
is conjugate to one of the subgroups in the set $\{I, IQ_1,\dots, IQ_{p^{m-1}},J,JQ_{1},\dots,JQ_{p^{m-1}}\}$.

Since $I\leqslant H$, we argue as in the proof of claim (i) that $g^{-1}Ig\cap H\neq\{1\}$ if and only if $IgH=IH$ as double cosets in $I \backslash G/H$. Thus
\[|\Ten_{H}^G(\Phi_{H})^I|=-1,\]
and similarly 
\[|\Ten_{H}^G(\Phi_{H})^J|=-1.\]
Suppose that $X\in \{IQ_1,\dots, IQ_{p^{m-1}},JQ_{1},\dots,JQ_{p^{m-1}}\}$. Then $|X\backslash G/H|=1$, and thus
\[|\Ten_{H}^G(\Phi_{H})^X|=|\Phi_{H}^{X\cap H}|=-1,\]
since $X\cap H$ is either $I$ or $J$.

In summary, $\Ten_H^G(\Phi_H)$ is the element of $B^\times(G)$ characterized by $|\Ten_H^G(\Phi_H)^X|=-1$ if and only if
$X$ is conjugate to one of the subgroups in the set \\$\{I, IQ_1,\dots, IQ_{p^{m-1}},J,JQ_{1},\dots,JQ_{p^{m-1}}\}$. 

Now consider the element
\[\omega=\sum_{i=1}^{p^{m-1}}\Inf_{G/Q_i}^G(\Phi_{G/Q_i})\in B^\times (G)\]
Then $|\omega^{IQ_i}|=|\omega^{JQ_i}|=-1$ for all $i=1,\dots,p^{m-1}$. Also, $|\omega^{I}|=|\omega^{J}|=(-1)^{p^{m-1}}=-1$. For any other subgroup $X\leqslant G$
that is not conjugate with $I,J, IQ_i$, or $JQ_i$ for $i=1,\dots, p^m$, it is easy to verify that $|\omega^X|=1$.
Thus $|\omega^X|=|\Ten_H^G(\Phi_H)^X|$ for any subgroup $X\leqslant G$, which implies that $\omega=\Ten_H^G(\Phi_H)$. In particular, we have shown that 
$\Ten_H^G(\Phi_H)$ is an element of $\sum_{1\neq K\unlhd G}\Im\Inf_{G/K}^G$.\\

\noindent{\bf Claim (iv):} \emph{If $|N_2|>2$ then $\Ten_{H_I}^G(\Phi_{H_I})$ and $\Ten_{H_J}^G(\Phi_{H_J})$ are the trivial element in $B^\times(G)$.}\\

We prove this for $\Ten_{H_I}^G(\Phi_{H_I})$ and note that the proof is nearly the same for $\Ten_{H_J}^G(\Phi_{H_J})$. Fix $X\leqslant G$.
It suffices to show that $|\Ten_{H_I}^G(\Phi_{H_I})^X|=1$.
Notice that since $|N_2|>2$, this implies that $H_I\cap N$ has no complement in $N$. 
So if there is an element $g\in G$ such that $^gX\cap H_I$ is noncentral of order $2$ in $H_I$,  then $X\cap N$ is trivial, hence $X$ is noncentral of order $2$ in $G$.
If $|\Ten_{H_I}^G(\Phi_{H_I})^X|=-1$, then there is such a $g\in G$, hence $X$ is noncentral of order $2$.
However, by $(3)$ we can verify that $|\Ten_{H_I}^G(\Phi_{H_I})^I|=1$ and $|\Ten_{H_I}^G(\Phi_{H_I})^J|=1$ (for the same reasons as in claim (ii)). 
This means that $|\Ten_{H_I}^G(\Phi_{H_I})^X|=1$ and we are done.\\\\

Claims (i) and (ii) prove that if $|N_2|=2$ or $N$ has an element of order $p^2$ for some odd prime $p$, then $G$ is not residual. On the other hand
when $|N_2|\neq 2$ and $N$ has no element of order $p^2$, claims (iii) and (iv) prove that 
$\Ten_H^G(\Phi_H)$ is contained in $\sum_{1\neq K\unlhd G}\Im\Inf_{G/K}^G$, for all pseudodihedral maximal subgroups $H$. Therefore,
applying Proposition~\ref{tensoringReduction} implies $G$ is residual with respect to $B^\times$.

\end{proof}

Consequentially, we get the following theorem.

\begin{theorem}\label{residualGroups}
Let $\mathcal{R}$ denote the set of finite groups consisting of the trivial group, dihedral groups $D_{2p}$
where $p$ is prime and $p\equiv 1\, (\text{mod} \,\,4)$, and all pseudodihedral groups $C_2\ltimes N$, excluding $D_8$, where $N$ is not simple, $|N_2|\neq2$, and $N_{2'}$ has no element
of order $q^2$, for any odd prime $q$. Then $\mathcal{R}$ is a complete set of residuals for $B^\times$ in $\C'$. Moreover, for any $G\in C'$,
\[B^\times(G)=\sum_{\underset{H\prec G}{H\in \mathcal{R}}}B(G,H)\Phi_H.\]

\end{theorem}

\begin{proof}
That $\mathcal{R}$ is a complete set of residuals follows by Remark~\ref{D8}, Proposition~\ref{residImpliesQuasiD}, and Proposition~\ref{residualGroupsAbIndex2}.
The last statement follows from Proposition~\ref{completeSetRes} and Lemma~\ref{PhiRestriction}.
\end{proof}

\section{Subfunctors of $B^\times$}

The goal of this section is to parametrize the subfunctors of $B^\times$ over $\C'$. We can do this in terms of the residual groups
found in the previous section. For the rest of the section, we let $\mathcal{R}_{\C',B^\times}$ be a complete set of residuals for $B^\times$ in $\C'$.

\begin{nota}\label{subsOfBx}
For any $I\subset \mathcal{R}_{\C',B^\times}$, set $F_I\subseteq B^\times$ to be the subfunctor generated by $\{\Phi_X|X\in I\}$. In other words, $F_I$ is
the subfunctor of $B^\times$ such that, for any $G\in \C'$
\[F_I(G)=\sum_{X\in I}B(G,X)\Phi_X,\]
if $I$ is nonempty and $F_I$ is the trivial biset functor if $I$ is empty. Note that we may construct $F_I\subseteq B^\times\in \F_\D$, where $\D$ is any replete subcategory
of $\C$.
\end{nota}

Recall that when $I=\mathcal{R}_{\C',B^\times}$, Theorem~\ref{residualGroups} implies $F_I=B^\times\in\mathcal{F}_{\C'}$.
If $I\subset\mathcal{R}_{\C',B^\times}$, we denote 
\[\overline{I}:=\{X\in \mathcal{R}_{\C',B^\times}|X\prec H\in I\},\]
i.e, $\overline{I}$ consists of the elements of $\mathcal{R}_{\C',B^\times}$ which are isomorphic to subquotients of elements of $I$. We
say $\overline{I}$ is the \emph{residual subquotient closure} of $I$.
If $I=\overline{I}$ then
$I$ is said to be \emph{closed under residual subquotients}. Recall that $\mathcal{R}_{\C',B^\times}$ consists of the trivial group and certain
pseudodihedral groups. Thus if $X\in I$ then every quotient of $X$ is isomorphic to a subgroup of $X$. So if we had defined $\overline{I}$ to be the subset of $\mathcal{R}_{\C',B^\times}$ consisting of elements isomorphic to subgroups of elements from $I$, it would
result in the same set.

\begin{theorem}\label{subParam}
Suppose $F$ is a subfunctor of $B^\times\in \F_{\C'}$. Set
\[I_F:=\{X\in \mathcal{R}_{\C',B^\times}|F(X)=B^\times(X)\},\]
then $F_{I_F}=F$ and $I_F=\overline{I_F}$. 
Moreover, if $\mathcal{A}=\{\overline{J}|J\subset \mathcal{R}_{\C',B^\times}\}$ and $\mathcal{B}$ is the set of subfunctors of $B^\times$ over $\C'$, let
$a:\mathcal{A}\to \mathcal{B}$ be the map sending $J\mapsto F_J$ and $b:\mathcal{B}\to \mathcal{A}$ the map sending $F\to I_F$, then $a$ and $b$ are isomorphisms of posets,
inverse to each other.
\end{theorem}

\begin{proof} If $F$ is trivial, then $I$ is empty and this case is clear. 
Suppose $F$ is nontrivial, there exists $X\in \C'$ and $u\in F(X)\subset B^\times(X)$, such that $u$ is nontrivial. 
Hence, there is a $K\leqslant X$, such that $|u^K|=-1$.
This implies that $\Def^K_{K/K}\circ\Res_K^X(u)=\Phi_{K/K}\in B(K/K,X)F(X)\subset F(K/K)\cong F(\{1\})$. Further, $F(\{1\})\subset B^\times(\{1\})$. 
However, since $\Phi_{\{1\}}\in B(\{1\})$ generates $B^\times(\{1\})$, we have $F(\{1\})=B^\times(\{1\})$. Thus $I$ contains the trivial group. 
Moreover, by Lemma~\ref{abelianGeneratedBy-1}, we have $F_I(X)=F(X)=B^\times(X)$ for abelian $X$.

It is clear that $F_I\subseteq F$. 
Fix a nonabelian object $G\in \C'$. We start with the case when $G$ has an abelian subgroup of index $2$ and show $F_I(G)=F(G)$. In this case,
Proposition~\ref{normChar} implies it is sufficient to show $\partial F_I(G)=\partial F(G)$, since having an abelian subgroup of index at most
$2$ is closed under taking quotients. For any object $X\in\C'$, we have
\[\partial F_I(X)\leqslant\partial F(X)\leqslant \partial B^\times(X).\] 
We also know that $\partial B^\times(G)$ is either trivial or has $\bF_2$-dimension $1$, by 
Lemma~\ref{PhiLemma} and Proposition~\ref{abelianSubgroup}. So we only need to show that when $\partial F(G)$ is nontrivial 
$\partial F_I(G)$ is also nontrivial. 

When $\partial F(G)$ is nontrivial, $\partial B^\times(G)$ is nontrivial. Moreover, since $G\neq \{1\}$, if $\partial F(G)$ is nontrivial, then
Proposition~\ref{abelianSubgroup} implies $G$ is pseudodihedral and $\Phi_G\in F(G)$. 
Thus, Proposition~\ref{PhiRestriction} tells us $B^\times(S)=B(S,G)\Phi_G\subset F(S)$, for any $S\prec G$. This implies
$F(S)=B^\times(S)$, for any $S\prec G$. So $I$ contains an isomorphic copy of any subquotient of $G$ that is residual with respect to $B^\times$. 
Hence, $\Phi_S\in F_I(S)$ for all $S\prec G$ that are residual with respect to $B^\times$. Thus, applying Theorem~\ref{residualGroups}
\[F(G)=B^\times(G)=\sum_{S}B(G,S)\Phi_{S}\subset F_I(G)\]
where $S$ runs over all the subquotients in $G$ which are residual with respect to $B^\times$. It follows that $F(G)=F_I(G)$ and in particular $\partial F(G)=\partial F_I(G)$. 

We have so far shown that $F(G)=F_I(G)$, whenever $G\in\C'$ has an abelian subgroup of index at most $2$.
If $G$ is a general object of $\C'$, then there is a normal subgroup $N\unlhd G$ with odd index in $G$, containing an abelian subgroup of index at most $2$. 
By the above
argument $F_I(N)=F(N)$. Since $F\subseteq B^\times$, the isomorphism in Proposition~\ref{oddIndexIso} given by $\Ten_N^G:B^\times(N)^{G/N}\to B^\times(G)$ 
restricts to an isomorphism $F(N)^{G/N}\to F(G)$. 
Since $F_I(N)^{G/N}=F(N)^{G/N}$, we have $F_I(G)=F(G)$.

That $I=\overline{I}$ follows from Proposition~\ref{PhiRestriction}. What is left is to establish that if $I\subset R_{\C',B^\times}$ is closed under residual subquotients,
and
\[I'=\{X\in \mathcal{R}_{\C',B^\times}|F_I(X)=B^\times(X)\},\]
then $I=I'$. Again, Proposition~\ref{PhiRestriction} implies that $I\subset I'$. If $I$ is empty, this is clear. 
Suppose $I$ is not empty and for the sake of contradiction that there is some $G\in I'$, such that $G\not \in I$.
Hence
\[B^\times(G)=F_I(G)=\sum_{X\in I}B(G,X)\Phi_X.\]
However, since $G$ is not isomorphic to a subquotient of any element of $I$, it follows from Remark~\ref{elementaryBisetPresentation}
that for any $\varphi\in B(G,X)$, we have $\varphi(\Phi_X)\in (B^\times)^<(G)$. But this implies that $B^\times(G)=(B^\times)^<(G)$,
which is a contradiction. Hence $I=I'$. It is easy to verify that this bijection is an isomorphism of posets.
\end{proof}

\begin{rmk}\label{subParamGen}
Theorem~\ref{subParam} has a formal generalization. If $\D$ is any replete subcategory of $\C$ contained in $\C'$, and $\mathcal{R}_{\D,B^\times}$ is a complete
set of residuals for $B^\times$ in $\D$ (for example, we can choose $\mathcal{R}_{\D,B^\times}$ to be the set of $G\in \mathcal{R}$ such that $G$ is an object of $\D$, where 
$\mathcal{R}$
is the set from Theorem~\ref{residualGroups}), then the lattice of subfunctors of $B^\times\in \F_{\D}$ is isomorphic as a poset to $\{\overline{J}|J\subset \mathcal{R}_{\D,B^\times}\}$.
The proof is the same, with the exception of using the symbol ``$\D$" whenever there is a ``$\C'$".
\end{rmk}

In Theorem $9.5$ of \cite{B2}, Bouc characterized the subfunctors of $B^\times$ over the class of $2$-groups (and through inflation, nilpotent groups).
In our terminology, Bouc showed that the lattice of subfunctors of 
$B^\times$ is uniserial and the nontrivial proper subfunctors are in bijection with the sets $\overline{\{D_{2^n}\}}$ for $n>3$. Together with Bouc's
result, Theorem~\ref{subParam} can be used to show the structure of the lattice of subfunctors of $B^\times$ over the full subcategory of $\C$, whose objects are nilpotent groups
or groups in $\C'$. Furthermore, it is easy to verify that the lattice of subfunctors of $B^\times$ is not uniserial over $\C'$. 
Additionally, Theorem~\ref{subParam} gives us a sufficient condition for when the biset functor $B^\times$, defined on a replete subcategory $\D\subset \C$ has uncountably
many subfunctors.

\begin{cor}\label{uncountableSubs}
Consider the following sets of groups: 
\[\mathcal{S}_0:=\{D_{2p}|\,\,p \text{ prime },\,\,p\equiv 1\, (\text{mod} \,\,4)\},\]
\[\mathcal{S}_1:=\{C_2\ltimes(C_p\times C_p)|\text{ $C_2\ltimes(C_p\times C_p)$ is pseudodihedral }, \\\,\,p \text{ prime }, \,\,p\equiv 3\, (\text{mod} \,\,4)\},\]
\[\mathcal{S}_2:=\{D_{2pq}|\,\,p\neq q; \,\,p, q \text{ prime }; \,\, p,q \equiv 3\, (\text{mod} \,\,4)\},\]
and
\[\mathcal{S}_r:=\{D_{2rp}|\,\,p \text{ prime }, \,\,p\equiv 3\, (\text{mod} \,\,4)\},\]
where $r=4$ or $r$ is a prime congruent to $1$ modulo $4$.
If $\D$ is a replete subcategory of $\C$ containing infinitely many objects from one of the sets $S_0,S_1,S_2$ or $S_r$, where $r=4$ or 
$r$ is a prime congruent to $1$ modulo $4$, then
$B^\times\in \F_{\D}$ has uncountably many subfunctors.
\end{cor}

\begin{proof}
Let $\mathcal{R}$ be the complete set of residuals with respect to $B^\times$ from Theorem~\ref{residualGroups}. 
Let $\{G_p\}_{p\in\mathcal{K}}$ be the assumed set of groups from the statement, whose elements are objects of $\D$, indexed by the infinite set $\mathcal{K}$. 
For any subset $\Pi\subset\mathcal{K}$, we denote 
\[I_{\Pi}=\overline{\{G_p\}_{p\in\Pi}}\subset \mathcal{R}.\]
For any two subsets $\Pi,\Pi'\subset \mathcal{K}$, we have $I_\Pi=I_{\Pi'}$ if and only if $\Pi=\Pi'$. It follows from Remark~\ref{subParamGen} that
$F_{I_\Pi}=F_{I_{\Pi'}}$ if and only if $I_{\Pi}=I_{\Pi'}$.
\end{proof}

In fact, the sufficient condition in Corollary~\ref{uncountableSubs} is also a necessary condition for any replete subcategory of $\C$ contained in $\C'$. However, before we
give its proof, we need a technical lemma about countability. We include a proof of it for completeness, though the reader may wish to skip it. 

Before we state the following lemma, we clarify our terminology. Consider a set $\mathcal{I}$ of $n$-tuples with entries from $\bN$ with the property that, 
if $(a_1,\dots, a_n)\in \mathcal{I}$ and $(b_1,\dots, b_n)$ is any other $n$-tuple with entries in $\bN$ such that $b_i\leq a_i$ for each $i=1,\dots, n$,
then $(b_1,\dots, b_n)\in \mathcal{I}$. In this case, we say that $\mathcal{I}$ is \emph{closed under the product ordering from below}. In the lemma, 
by $\bN$, we mean the subset of $\bZ$ consisting of the positive integers and $0$.

\begin{lemma}\label{prodClosedCountable}
Let $n$ be a positive integer. There are only countably many subsets of $\bN^n$ that are closed under the product ordering from below.
\end{lemma}

\begin{proof}
We will prove this by showing that a nonempty subset of $\bN^n$, closed under the product ordering from below, can be parametrized by a finite subset
of  $(\bN_\infty)^n$, where $\bN_\infty:=\bN\cup\{\infty\}$ and $\infty$ is some maximal element, with respect to the regular ordering. Since
$(\bN_\infty)^n$ is countable, and the finite subsets of a countable set form a countable set, the result will follow.

For each $\mathcal{I}\subset\bN^n$ that is closed under the product ordering from below, there is a unique 
$\mathcal{I}^*=\mathcal{I}\cup\mathcal{I}_\infty\subset (\bN_\infty)^n$, where 
\[\mathcal{I}_\infty=\{{\bf x}\in (\bN_\infty)^n-\bN^n|\,\, {\bf y}<{\bf x}\,\, \implies\,\,{\bf y}\in \mathcal{I}, \forall {\bf y}\in \bN^n\}.\]
Note that if $\mathcal{I}$ is finite, then $\mathcal{I}_\infty=\emptyset$ and $\mathcal{I}^*=\mathcal{I}$. Furthermore, it is straightforward to verify that 
$\mathcal{I}^*$ is closed under the product ordering
from below in $(\bN_\infty)^n$ and $\mathcal{I}^*\cap \bN^n=\mathcal{I}$. Moreover, for every $(x_1,\dots, x_n)\in \mathcal{I}^*$, there is a maximal element
$(a_1,\dots, a_n)\in\mathcal{I}^*$ such that $(x_1,\dots, x_n)\leq (a_1,\dots, a_n)$.

Let $M$ be a set of elements, from $(\bN_\infty)^n$, that are pairwise noncomparable with respect to the product ordering. Set $\mathcal{I}_M$ to be the closure,
from below, with respect to the product ordering. If $\mathcal{I}\subset  \bN^n$ is closed under the product ordering from below, and $M$ is the set of maximal elements
in $\mathcal{I}^*$, our construction guarantees that $\mathcal{I}_M=\mathcal{I}^*$. Thus, it suffices to show that $\mathcal{I}^*$ has only finitely many maximal elements.

We do this inductively on $n$. If $n=1$, this is clear, since every nonempty proper subset of $\bN$, closed under the product ordering from below has one 
maximal element and $\bN^*=\bN_\infty$ also has one maximal element. 
For the general case, suppose $n>1$. Let $\mathcal{I}\subset\bN^n$ be a proper subset,
closed under the product ordering from below, then there is a maximal $a\in \bN$, such that $(a,\dots,a)\in \mathcal{I}^*$. 
For each $i\in \{1,\dots, n\}$ and each $m\in \{1,\dots, a\}$, let $I^m_i$ be the subset of $\mathcal{I}^*$ consisting of all $n$-tuples with $m$ in the $i$th component.
Notice that every element of $\mathcal{I}^*$ is contained in at least one $I^m_i$, for some $i\in\{1,\dots, n\}$ and $m\in \{0,\dots, a\}$.
Otherwise, there is an element $(x_1,\dots, x_n)\in \mathcal{I}^*$ where $(x_1,\dots, x_n)\not\in I^m_i$, for all $i=1,\dots,n$
and all $m=0,\dots, a$, hence $x_i>a$ for all $i$. But this means that $(a+1,\dots, a+1)\in \mathcal{I}^*$, which is a contradiction. 
Furthermore, for any $i=1,\dots,n$
and all $m=0,\dots, a$, the image of $I^m_i$ in $(\bN_\infty)^{n-1}$, created by deleting the $i$th entry in each element of $I^m_i,$ 
is closed under the product ordering from below. Thus, by induction, $I^m_i$ has finitely many maximal elements, all with the entry $m$ in the $i$th component. If
we denote by $M^m_i$ the set of maximal elements of the set $I^m_i$ for every $i=1,\dots,n$
and all $m=0,\dots, a$, and denote
\[M_{\mathcal{I}}=\bigcup_{1\leq i\leq n,0\leq m\leq a}I^m_i,\]
then $M_{\mathcal{I}}$ is a finite set. Moreover, $M_{\mathcal{I}}$ will contain all the maximal elements of $\mathcal{I}^*$, and the result follows.

\end{proof}

\begin{theorem}\label{converseUncountableSubs}
Suppose $\D$ is a replete subcategory of $\C'$. The biset functor $B^\times$ over $\D$ has uncountably many subfunctors if and only if $\D$ has infinitely many
objects from one of the sets $S_0,S_1,S_2$ or $S_r$, where $r=4$ or 
$r$ is a prime congruent to $1$ modulo $4$, from Corollary~\ref{uncountableSubs}.
\end{theorem}

\begin{proof}
The "if" direction follows from Corollary~\ref{uncountableSubs}. What is left is to prove that if $\D$ contains only finitely many objects from each of the sets described in Corollary~\ref{uncountableSubs}, then the lattice of subfunctors of $B^\times$ over $\D$ is at most countable. 
Suppose $\D$ contains only finitely many objects from each set described in Corollary~\ref{uncountableSubs}. 
Let $\mathcal{U}$ denote the objects contained in $\D$ from the union of all these sets. 
Since $\D$ contains only finitely many objects from $S_0$, it follows that $\mathcal{U}$ must be finite.
Notice that if $G\in\C'$ is residual with respect to $B^\times$ and $p$ is an odd prime dividing the order of $|G|$, then $G$ has a
subquotient (in fact, a subgroup) $X$ isomorphic to a group from one of the sets in Corollary~\ref{uncountableSubs}, such that $p||X|$.
Thus, by the finiteness of the set $\mathcal{U}$, there are odd primes $p_1,\dots,p_r$ such that, for any $X\in \D$, residual with respect to $B^\times$,
we have natural numbers $n_0,n_1\dots, n_r$ and a pseudodihedral group 
\[G_{(n_0,n_1,\dots, n_r)}:=C_2\ltimes(C_{2^{n_0}}\times (C_{p_1})^{n_1}\times\cdots\times (C_{p_r})^{n_r}),\]
with $X\prec G_{(n_0,n_1,\dots, n_r)}$.

By Remark~\ref{subParamGen}, we may assume that every object in $\D$ is a subquotient of some object which is residual with respect to $B^\times$,
since this does not change the poset structure. Suppose $\D'$ is a full subcategory of $\C$ consisting of subquotients of the groups $G_{(n_0,n_1\dots, n_r)}$,
for any $(r+1)$-tuple of natural numbers $(n_0,n_1\dots, n_r)$. Then $\D\subset \D'$ and the restriction functor $\F_{\D'}\to \F_{\D}$ induces a surjective
morphism between the poset of subfunctors of $B^\times$ over $\D'$ and the poset of subfunctors of $B^\times$ over $\D$. Thus, it suffices to show that 
$B^\times$ over $\D'$ has only countably many subfunctors.

Let $\mathcal{R}$ be the complete set of residuals with respect to $B^\times$ in $\C'$ form Theorem~\ref{residualGroups}. 
To see that  $B^\times$ over $\D'$ has only countably many subfunctors, we first set $\mathcal{R}_{\D',B^\times}$ be the complete set of residual with respect to $B^\times$ in $\D'$, consisting of
elements of $\mathcal{R}$ that are in $\D'$. Using Theorem~\ref{subParam}, the result follows if the set $\{\overline{J}|J\subset\mathcal{R}_{\D',B^\times}\}$ is countable.
To see this, let $J\subset\mathcal{R}_{\D',B^\times}$ and consider the set $\{(a_0,a_1,\dots,a_r)\in \bN|G_{(a_0,a_1,\dots, a_r)}\prec G_{(n_0,n_1,\dots, n_r)}\in \overline{J}\}$.
This set is closed under the product ordering on $\bN^{r+1}$. Moreover, this induces an injection into the set of subsets of $\bN^{r+1}$ that are closed under the product ordering.
The result follows by Lemma~\ref{prodClosedCountable}.
\end{proof}

It is natural to ask if we can relax the condition in Theorem~\ref{converseUncountableSubs}, for $\D$ to be a subcategory of $\C'$.

\begin{ques}
Let $\D$ be any replete subcategory of $\C$. 
Is the condition in Corollary~\ref{uncountableSubs} a necessary condition for the biset functor $B^\times$ over $\D$ to have uncountably many subfunctors?
\end{ques}

The next result generalizes Corollary~\ref{oddIndexCor}.

\begin{prop}\label{computingSubfuncts}
Suppose $I\subset \mathcal{R}_{\C',B^\times}$ is closed under residual subquotients. Let $G$ be an
object of $\C'$, and $H$ a normal subgroup in $G$ with odd index, such that $H$ has an abelian subgroup of index at most $2$. Let $\mathcal{N}$
be the set of all normal subgroups $N\unlhd H$, such that $H/N$ is trivial, cyclic of order $2$, or pseudodihedral, such that every suquotient of $H/N$ 
that is residual with respect to $B^\times$
is isomorphic to an element in $I$. Set $\mathcal{L}=\{\Inf_{H/N}^H\Phi_{H/N}\}_{N\in \mathcal{N}}$, then $\mathcal{L}$ is a basis for
$F_I(H)$ and $G/H$ acts on $\mathcal{L}$. If $L_1,\dots,L_k$ denote the orbit sums of this action, then $\{Ten^G_HL_1\}_{i=1}^k$
is a basis for $F_I(G)$.
\end{prop}

\begin{proof}
The only thing we need to prove is that $\mathcal{L}$ is a basis of $F_I(H)$, the rest follows by replacing $B^\times$ by $F_{I}$
in the proof of Corollary~\ref{oddIndexCor}. We can use Proposition~\ref{normChar} to decompose 
\[F_I(H)\cong \bigoplus_{N\unlhd H}\partial F_I(H/N).\]
We have $F_I(H/N)\subset B^\times(H/N)$, hence 
$\partial F_I(H/N)\subset \partial B^\times(H/N)$. Thus $\partial F_I(H/N)$ is trivial or $\partial F_I(H/N)$ is generated by $\Phi_{H/N}$. 
If $X$ is a pseudodihedral groups such that every 
subquotient residual with respect to $B^\times$ is isomorphic to an element of $I$, then $B^\times(X)=F_I(X)$. 
Thus $\langle \Inf_{H/N}^H\Phi_{H/N}\rangle_{N\in \mathcal{N}}\subset F_I(H)$. 

Now, suppose $M\unlhd H$ and $H/M$ has a subquotient
$S$, such that $S$ is isomorphic to an element $X\in \mathcal{R}_{\C',B^\times}-I$. For the sake of contradiction, suppose $\Phi_{H/M}\in\partial F_I(H/M)$.
Then by Proposition~\ref{PhiRestriction}, 
\[B^\times(X)\cong B^\times(S)=B(S,H/M)\Phi_{H/M}\subseteq F_I(S)\cong F_I(X),\]
and by dimension, we have $B^\times(X)=F_I(X)$. Hence $X\in I$, which is a contradiction.
Therefore, $\langle \Inf_{H/N}^H\Phi_{H/N}\rangle_{N\in \mathcal{N}}$ must be all of  $F_I(H)$
\end{proof}

In general, computing the dimension of the evaluations of simple biset functors is difficult. In the next section the explicit form of 
Proposition~\ref{computingSubfuncts} will give us an easy way to compute the dimension of $S_{G,\bF_2}(H)$
when $G,H\in \C'$ and $G$ is residual with respect to $B^\times$. 

We end this section by detailing which simple biset functors show up
as composition factors of $B^\times\in \F_{\C'}$. Before we state the following theorem, we clarify a technicality. Let $\D$ be any replete subcategory of $\C$.
Note that $B^\times(G)$ is always an elementary abelian $2$-group, thus we may view $B^\times$ as a functor in $B^\times\in \F_{\D,\bF_2}$. So there is
no ambiguity when we discuss the multiplicity of composition factors of $B^\times$.

\begin{theorem}\label{compFactorsBX}
The composition factors of $B^\times\in \F_{\C'}$ are parametrized exactly by the simple biset functors $S_{G,\bF_2}$, where 
$G\in \mathcal{R}_{\C',B^\times}$. Furthermore, each composition factor has multiplicity $1$.
\end{theorem}

\begin{proof}
Let $F_{I'}\subseteq F_{I}\subseteq B^\times$ be two subfunctors of $B^\times\in \F_{\C'}$ such that $F_{I}/F_{I'}$ is simple, 
where $I$ and $I'$ are the corresponding subsets of $\mathcal{R}_{\C',B^\times}$ from Theorem~\ref{subParam}. We may then assume that
$F_{I}/F_{I'}\cong S_{G,V}$ where $G\in \C'$ and $V$ is a simple $\bF_2\text{Out}(G)$-module, such that $S_{G,V}(G)\cong V$.
 By 
Theorem~\ref{subParam}, $I'\subsetneq I$. Moreover, we claim that $I- I'$ is a one element set consisting of a group that
is isomorphic to $G$. Indeed, suppose $H_1,H_2\in I- I'$ such that $H_1\neq H_2$. We may assume that $H_1$ is not a subquotient of $H_2$.
Let $I''=\overline{I'\cup\{H_2\}}$. Thus $I'\subsetneq I''\subsetneq I$, which implies that $F_{I}/F_{I'}$ is not simple, by 
Theorem~\ref{subParam}. Furthermore, if $H$ is
the unique element in $I- I'$, then for any $X\in \C'$ such that $|X|<|H|$, we have $F_{I'}(X)=F_{I}(X)$, by Proposition~\ref{computingSubfuncts}. 
Thus $S_{G,V}(X)$ is trivial. Again by by Proposition~\ref{computingSubfuncts}, we have $F_{I'}(H)\subsetneq F_{I}(H)$.
So $H$ is a minimal object for $S_{G,V}$ and by Remark~\ref{simpleParam}, $H\cong G$. We can now assume $H=G$. Through evaluation
\[V\cong S_{G,V}(G)\cong F_{I}/F_{I'}(G)\cong B^\times(G)/(B^\times)^<(G),\]
has dimension $1$ as an $\bF_2$-space by Proposition~\ref{residImpliesQuasiD} and Proposition~\ref{PhiLemma}, since $G$ is residual.
Furthermore, $B^\times(G)/(B^\times)^<(G)$ is generated by the image of the element $\Phi_G\in B^\times(G)$. By the proof of Proposition $4.3.2$
in \cite{B1}, the action of $\text{Out}(G)$ on $V$ as an $\bF_2\text{Out}(G)$-module, is given by $f\cdot v=\Iso(f)(v)$ for any $f\in \text{Out}$ and $v\in V$.
Since $\Phi_G$ is the unique faithful element in $B^\times(G)$, it is fixed by $\Iso(f)$ for any $f\in \text{Out}(G)$. Thus $V$ is isomorphic to the
trivial $\bF_2\text{Out}(G)$-module.

That $S_{G,\bF_2}$ has multiplicity $1$ is easy to see by a simple construction.
Let $I=\overline{\{G\}}\subset \mathcal{R}_{\C',B^\times}$. Then choose 
$\{\{1\}\}=I_1\subsetneq I_2\subsetneq\cdots\subsetneq I_{n}\subsetneq I_{n+1}=I$ to be a filtration such that $|I_{i+1}- I_{i}|=1$ and $I_i$ is closed under residual subgroups, for each $i=1,\dots, n$. Note that this forces
$I_{n+1}- I_{n}=\{G\}$. Thus
\[F_{I_1}\subseteq F_{I_1}\subset F_{I_2}\subseteq F_{I_2}\subset\cdots \subseteq F_{I_n}\subset F_{I_{n+1}}\subseteq B^\times\]
is a composition series for $B^\times$ over $G$. It is easy to verify that $F_{I_{n+1}}/F_{I_n}$ is the only composition factor isomorphic to 
$S_{G,\bF_2}$, since $G\not\in I_i$ for any $i=1,\dots,n$. 
\end{proof}

\section{Some Applications}

\subsection{Computing $S_{G,\bF_2}(H)$ for $G$ residual with respect to $B^\times$}
\hfill\\

For this section, we again set $\mathcal{R}_{\C',B^\times}$ to be a complete set of residuals for $B^\times$ in $\C'$.

If $G$ and $H$ are finite groups and $k$
is a field, the following is a characterization due to Bouc of the $k$-module $S_{G,k}(H).$

\begin{prop}[\cite{B1}, proposition 4.4.6, page 71]\label{simpleModuleCharacterization}
Let $\D$ be a replete subcategory of $\C$. For any objects $G$ and $H$ of $\D$, let $kB_H(G)$ be the $k$-vector space with basis the set of conjugacy classes of sections $(T,S)$ of $G$ such that
$T/S\cong H$.

Then the module $S_{H,k}(G)$ is isomorphic to the quotient of $kB_H(G)$ by the radical of the $k$-valued symmetric
bilinear form on $kB_H(G)$ defined by
\[\langle(B,A)|(T,S)\rangle_H=|\{h\in B\backslash G/T|(B,A)-(\,^hT,\,^hS)\}|_k,\]
where $(B,A)-(T,S)$ means that 
\[B\cap S=A\cap T\quad\quad (B\cap T)A=B\quad\quad (B\cap T)S=T.\]
In particular, the $k$-dimension of $S_{H,k}(G)$ is equal to the rank of this bilinear form.
\end{prop}

It is generally difficult to determine the $k$-dimension of $S_{G,k}(H)$ in terms of other data about $G$ and $H$. 
However, we can apply the results from the the previous section to determine this for $S_{G,\bF_2}(H)$, where 
$G\in \mathcal{R}_{\C',B^\times}$ and $H\in \C'$.

\begin{theorem}\label{simpleComputation}
Suppose $G\in \mathcal{R}_{\C',B^\times}$ and $H\in \C'$. Set $r=\dim_{\bF_2}(F_I(H))$ and $l=\dim_{\bF_2}(F_J(H))$, where 
$I=\overline{\{G\}}$, $J=I-\{G\}$.
Then $\dim_{\bF_2}(S_{G,\bF_2}(H))=r-l$.
\end{theorem}

\begin{proof}
Notice that $\overline{J}=J$. Since $J\subset I$, Theorem~\ref{subParam} $F_I/F_J$ is simple. Proposition~\ref{computingSubfuncts} implies that $G$ is a 
minimal object for $F_I/F_J$.
Thus by Remark~\ref{simpleParam} and Theorem~\ref{compFactorsBX}, $F_I/F_J\cong S_{G,\bF_2}$ and the result follows.
\end{proof}

Theorem~\ref{simpleComputation}, together with Proposition~\ref{computingSubfuncts}, gives us a way to compute
$S_{G,\bF_2}(H)$, where $G\in \mathcal{R}_{\C',B^\times}$ and $H\in \C'$. To illustrate this, we cover the cases where $G$ is trivial or dihedral and $H$ is dihedral.

\begin{example}\label{computation1}
Consider the dihedral group $D_{2k}$, where $k\geq 3$. Let $p_1,\dots,p_r$ denote the prime divisors of $k$ that are congruent to $3$ modulo $4$.
Set $m_i$ to be the power of the $p_i$-part of $k$, for $i=1,\dots,r$ and $m_0$ the power of the $2$-part of $k$. We show that 
\[\dim_{\bF_2}(S_{\{1\},\bF_2}(D_{2k}))= \begin{cases}      
     2+m_1+\cdots+ m_r & \text{ if } \quad m_0=0\\
     4+2(m_1+\cdots +m_r) & \text{ if } \quad m_0=1\\
     5+2(m_1+\cdots +m_r) & \text{ if } \quad m_0>1     
   \end{cases}
\]
By Proposition~\ref{computingSubfuncts}, $\dim_{\bF_2}(S_{\{1\},\bF_2}(D_{2k}))$
is equal to the number of normal subgroups $N\unlhd D_{2k}$, such that $D_{2k}/N$ is trivial, cyclic of order $2$, or dihedral (nonabelian), and
there are no nontrivial subquotients of $D_{2k}/N$ residual with respect to $B^\times$. If $D_{2k}/N$ is dihedral and there are
no nontrivial subquotients of $D_{2k}/N$ residual with respect to $B^\times$, Theorem~\ref{residualGroups} implies
that $D_{2k}/N\cong D_{8}$ or $D_{2k}/N\cong D_{2l}$, where $l=2^\epsilon p^s$, where $p$ is a prime congruent to $3$ modulo $4$, and $\epsilon\in \{0,1\}$.
The following cases are straightforward to verify:
\begin{itemize}
\item If $m_0=0$, then there is one normal subgroup of $D_{2k}$ with trivial quotient, one normal subgroup of $D_{2k}$ of index $2$,
and $m_1+\cdots+ m_r$ normal subgroups $N\unlhd D_{2k}$, such that $D_{2k}/N$ is dihedral and there are no nontrivial 
subquotients of $D_{2k}/N$ residual with respect to $B^\times$.\\

\item If $m_0=1$, then there is one normal subgroup of $D_{2k}$ with trivial quotient, three normal subgroups of $D_{2k}$ of index $2$,
and $2(m_1+\cdots+ m_r)$ normal subgroups $N\unlhd D_{2k}$, such that $D_{2k}/N$ is dihedral and there are no nontrivial subquotients of $D_{2k}/N$ residual with respect to $B^\times$.\\

\item If $m_0>1$, then there is one normal subgroup of $D_{2k}$ with trivial quotient, three normal subgroups of $D_{2k}$ of index $2$,
and $2(m_1+\cdots+ m_r)+1$ normal subgroups $N\unlhd D_{2k}$, such that $D_{2k}/N$ is dihedral and there are no nontrivial subquotients of $D_{2k}/N$ residual with respect to $B^\times$.
\end{itemize}

Set $D_{2}\cong C_2$ and $D_{4}\cong C_2\times C_2$. 
It is straightforward to verify $\dim_{\bF_2}(S_{\{1\},\bF_2}(D_{2}))=2$ and $\dim_{\bF_2}(S_{\{1\},\bF_2}(D_{4}))=4$. 
There is then a function $s:\bN\to\bN$ defined as $s(n)=\dim_{\bF_2}(S_{\{1\},\bF_2}(D_{2n}))$.
Let $\phi$ denote Euler's $\phi$-function, and set $m$ to be the number of positive divisors 
$d|n$, such that
$\phi(d)\equiv 2\, (\text{mod} \,\,4)$. Recall that for any $x\in \bN$, $\phi(x)\equiv 2\, (\text{mod} \,\,4)$ if and only if $x=4$, $x$ is the power of a prime congruent to $3$ modulo $4$,
or $x$ is twice the power of a prime congruent to $3$ modulo $4$. Thus $s(n)=m+2$ if $n$ is odd and $s(n)=m+4$ if $n$ is even.
\end{example}

\begin{example}\label{computation2}
Suppose $D_{2n}$ is a dihedral
group residual with respect to $B^\times$. If $n$ is not a power of $2$, by Theorem~\ref{residualGroups} we can decompose
\[n=2^{n_0}p_1\cdots p_r\]
where $n_0\neq1$, and the $p_i$ pairwise coprime odd primes, for $i=1,\dots,r$. Moreover, if $n_0=0$ and $r=1$,
then $p_1 \equiv 1\, (\text{mod}\,\, 4)$. 

Let $k\geq 3$ and consider the dihedral group $D_{2k}$. Let $m_0$ denote the exponent of the
$2$-part of $k$.
If $n$ is a power of $2$ (recall that in this case $n\geq8$), then we set $m=1$. Otherwise, set $m=m_1\cdots m_r$, where $m_{i}$ denotes the power of the $p_i$-part of $k$, for $i=1,\dots,r$. 
We show that
\[\dim_{\bF_2}(S_{D_{2n},\bF_2}(D_{2k}))= \begin{cases} 
     0 & \text{ if } \quad n\nmid k \\ 
     m & \text{ if } \quad n|k; \,\, n_0\neq0 \text{ or } m_0=0 \\
     2m & \text{ if } \quad n|k;\,\, n_0=0 \text{ and } m_0\neq0
   \end{cases}
\]
If $n\nmid k$, then $D_{2n}$ does not show up as a subquotient of $D_{2k}$, hence by Remark~\ref{simpleParam}, $S_{D_{2n},\bF_2}(D_{2k})=0$.

Suppose $n|k$. Set $I=\overline{\{D_{2n}\}}$ and $J=I-\{D_{2n}\}$. Theorem~\ref{simpleComputation} and Proposition~\ref{computingSubfuncts}
tell us that $\dim_{\bF_2}(S_{D_{2n},\bF_2}(D_{2k}))$ is equal to the number of normal subgroups $N\unlhd D_{2k}$, such that every subquotient, residual with respect to 
$B^\times$,
of $D_{2k}/N$ is in $I$ and at least one is not in $J$. In other words, every subquotient, residual with respect to 
$B^\times$, of $D_{2k}/N$ is in $I$ and $D_{2n}$ is a subquotient of $D_{2k}/N$.
When $n_0\neq 0$, this means that $D_{2k}/N$ must be isomorphic to a dihedral group $D_{2l_N}$ such that 
\[l_N=2^{n_0}p_1^{n_1}\dots p_r^{n_r},\]
with $1\leq n_i\leq m_i$, $i=1,\dots, r$.
When $n_0=0$, this means that $D_{2k}/N$ must be isomorphic to a dihedral group $D_{2l_N}$
\[l_N=2^{\epsilon}p_1^{n_1}\dots p_r^{n_r},\]
with $1\leq n_i\leq m_i$, $i=1,\dots, r$ and $\epsilon\in \{0,1\}$. 
By properties of dihedral groups, there is a bijective correspondence between normal subgroups of $N\unlhd D_{2k}$ of this type and divisors $d|k$, such that $\frac{k}{d}=l_N$.
The formula for $\dim_{\bF_2}(S_{D_{2n},\bF_2}(D_{2k}))$ follows by counting such divisors of $k$.

\end{example}

\subsection{Surjectivity of the the exponential map $B\to B^\times$}
\hfill\\

For any finite group $G$, we define the exponential map 
\[\varepsilon_G:B(G)\to B^\times(G)\]
to be the group homomorphism induced by sending 
\[[G/H]\mapsto \Ten_{H}^G(-1),\]
for any $H\leqslant G$.

\begin{rmk}
It is straightforward to show that the collection of all these group morphisms $\varepsilon_G$ is a morphism $\varepsilon:B\to B^\times$ of biset functors. 
Yoshida proves this in Lemma $3.2$ of \cite{Yo}. Yal\c{c}{\i}n also discusses this map in Section $7$ of \cite{Ya}. 
The way we define it in this paper is in the spirit of
Bouc $9.8$ of \cite{B2}.
\end{rmk}

The next result is proved by Bouc $9.7$ of \cite{B2} for $p$-groups. It is likely known to experts but we include a proof for the reader, 
since we make use of it to determine when the exponential map is surjective for objects in $\C'$.

\begin{prop}\label{expoImage}
Let $G$ be a finite group, then $\Im(\varepsilon_G)\cong S_{1,\bF_2}(G)$ as $\bF_2$-modules.
\end{prop}

\begin{proof}
If $k=\bF_2$ and $H=\{1\}$, then in reference to Proposition~\ref{simpleModuleCharacterization} we can identify
$\bF_2B_{1}(G)$ with $\bF_2B(G)$ by associating $(K,K)$ to $[G/K]\in\bF_2B(G)$. Since $B^\times(G)$ is an $\bF_2$ vector space, 
$\varepsilon_G:B(G)\to B^\times(G)$ is equal to the map $\overline{\varepsilon}\circ\pi$, where $\pi:B(G)\to \bF_2B(G)$ is the projection
map and $\overline{\varepsilon}:\bF_2B(G)\to B^\times(G)$ is again the group homomorphism induced by sending $[G/H]\to \Ten_H^G(-1)$.
It suffices to show that the kernel of $\overline{\varepsilon}$ is equal to the radical of the symmetric bilinear form in Proposition~\ref{simpleModuleCharacterization}, on $\bF_2B_{1}(G)$. For $\bF_2B_{1}(G)$, this form is defined by
\[\langle (K,K)|(S,S)\rangle_1=|\{h\in K\backslash G/S\}|_{\bF_2},\]
since the requirement that $(K/K)-(\,^hS,\,^hS)$, is true for any $K,S\leqslant G$ and $h\in G$. Thus, if $K_1,\dots,K_n$
are subgroups of $G$ and $X=(K_1,K_1)+\cdots +(K_n,K_n)\in \bF_2B_{1}(G)$, then $X$ is in the radical our bilinear form if
and only if 
\[\sum_{i=1}^n|\{h\in S\backslash G/K_i\}|\] 
is even for all subgroups $S\leqslant G$. Similarly by the formula in Remark~\ref{tensorInductionComp}
the element associated to $X$ in $\bF_2B(G)$ is in the kernel of of $\varepsilon_G$ if and only if 
\[\sum_{i=1}^n|\{h\in S\backslash G/K_i\}|\] 
is even for all subgroups $S\leqslant G$.
\end{proof}

We can now determine when the exponential map $\varepsilon_G:B(G)\to B^\times(G)$ is surjective for $G\in \C'$.

\begin{theorem}\label{surjectiveExpo}
Suppose $G\in \C'$. Let $N\unlhd G$ be a normal subgroup with odd index having an abelian subgroup of index at most $2$. 
Let $\mathcal{R}_{\C',B^\times}$ be a complete set of residuals for $B^\times$ in $\C'$. Then $\varepsilon_G:B(G)\to B^\times(G)$
is surjective if and only if no nontrivial quotient of $N$ is isomorphic to an element of $\mathcal{R}_{\C',B^\times}$.
 \end{theorem}

\begin{proof}
By Proposition~\ref{computingSubfuncts}, we have $S_{1,\bF_2}(G)=B^\times(G)$ if and only if no nontrivial subquotient of a 
pseudodihedral quotient of $N$ is isomorphic to an element of 
$\mathcal{R}_{\C',B^\times}$. 

Suppose $M\unlhd N$, such that $N/M$ is pseudodihedral and $N/M$ has a nontrivial subquotient isomorphic to an element of 
$\mathcal{R}_{C',B^\times}$. Lemma~\ref{quotientSubConnection} implies that $N/M$ has a quotient isomorphic to a nontrivial element of $\mathcal{R}_{C',B^\times}$,
thus $N$ has a nontrivial quotient isomorphic to an element of $\mathcal{R}_{C',B^\times}$. It follows that $N/M$ will have a subquotient isomorphic to an element of 
$\mathcal{R}_{C',B^\times}$ if and only if $N/M$ has a quotient isomorphic to an element of $\mathcal{R}_{C',B^\times}$.

Hence $S_{1,\bF_2}(G)=B^\times(G)$ if and only if no nontrivial quotient of $N$ is isomorphic to an element of 
$\mathcal{R}_{\C',B^\times}$.
The result follows from Proposition~\ref{expoImage}.
\end{proof}

A particularly nice case of the above result is when we consider dihedral groups. In the following result, $\phi:\bN\to \bN$ denotes Euler's $\phi$-function.
 
\begin{cor}\label{surjectiveExpoDihedral}
For any positive integer $n\geq 3$, $\varepsilon_{D_{2n}}:B(D_{2n})\to B^\times (D_{2n})$ is surjective if and only if $\phi(n)\equiv 2\, (\text{mod} \,\,4).$
\end{cor}
\begin{proof}
Proposition~\ref{expoImage}, implies that $\varepsilon_{D_{2n}}$ will be surjective if and only if 
\[\dim_{\bF_2}(S_{\{1\},\bF_2}(D_{2n}))=\dim_{\bF_2}(B^\times(D_{2n})).\] 
Recall that by Example~\ref{unitGroupComputations}, if we use $d:\bN\to \bN$, to denote the function that counts the number of positive divisors of any $n\in\bN$, then
$\dim_{\bF_2}(B^\times(D_{2n}))=d(n)+1$  if $n$ is odd, and $\dim_{\bF_2}(B^\times(D_{2n}))=d(n)+2$ if $n$ is even. By Example~\ref{computation1}, the function 
$s:\bN\to \bN$, where $s(n)=S_{\{1\},\bF_2}(D_{2n})$, is equal to the number of positive divisors $d|n$, such that 
$\phi(d)\equiv 2\, (\text{mod} \,\,4)$, plus $2$ or $4$, depending on whether $n$ is odd or even, respectively. If $n\geq 3$, one can verify that 
$s(n)=\dim_{\bF_2}(B^\times(D_{2n}))$ if and only if $\phi(n)\equiv 2\, (\text{mod} \,\,4)$.

\end{proof}


\begin{thebibliography}{HNTY}

\bibitem[1]{Ba} M. Baumann. Le foncteur de bi-ensembles des modules de $p$-permutation. PhD diss., \emph{\'Ecole Polytechnique F\'ed\'erale De Lausanne}, 5338, 2012.

\bibitem[2]{B2} S. Bouc. The functor of units for Burnside rings for $p$-groups. \emph{Comm. Math. Helv.,}
82:583-615, 2007.

\bibitem[3]{B1} S. Bouc. Biset functors for finite groups. Springer, Berlin, 2010.

\bibitem[4]{G} D. Gluck. Idempotent formula for the Burnside ring with applications to the $p$-subgroup simplicial complex. \emph{Illinois J. Math.}, 25:63-67, 1981.

\bibitem[5]{Ma} T. Matsuda. On the unit groups of Burnside rings. \emph{Japan. J. Math.}, 8(1):71-93, 1982.

\bibitem[6]{Ma2} T. Matsuda. On the unit groups of Burnside rings of finite groups. \emph{J. Math. Soc. Japan,}, 35(1):345-354, 1983.

\bibitem[7]{We} P. Webb. Two Classifications of simple Mackey functors with applications to group cohomology and the decomposition of classifying spaces. \emph{J. Pure Appl. Algebra}, 88 (issues 1-3):265-304, 1993. 

\bibitem[8]{Ya} E. Yal\c{c}{\i}n. An induction theorem for the unit groups of Burnside rings of $2$-groups. \emph{J. of Algebra}, 289:105-127, 2005.

\bibitem[9]{Yo2} T. Yoshida. Idempotents in the Burnside rings and Dress induction theorem. \emph{J. Algebra}, 80:90-105, 1983.

\bibitem[10]{Yo} T. Yoshida. On the unit groups of Burnside rings. \emph{J. Math. Soc. Japan,} 42(1):31-64, 1990.

\end{thebibliography}
\end{document}